\documentclass[a4paper, reqno, 12pt]{amsart}

\usepackage[usenames,dvipsnames]{color}
\usepackage[utf8]{inputenc}
\usepackage[T1]{fontenc}
\usepackage{amsthm,amsfonts,amssymb,amsmath,amsxtra}
\usepackage[all]{xy}
\SelectTips{cm}{}
\usepackage{xr-hyper}
\usepackage[colorlinks=
   citecolor=Black,
   linkcolor=Red,
   urlcolor=Blue]{hyperref}
\usepackage{verbatim}
\usepackage{tikz,tikz-cd}
\usepackage{graphicx}
\usepackage{adjustbox}
\usetikzlibrary{positioning}

\usepackage[margin=1.25in]{geometry}
\usepackage{mathrsfs}

\RequirePackage{xspace}
\RequirePackage{etoolbox}
\RequirePackage{varwidth}
\RequirePackage{enumitem}
\RequirePackage{tensor}
\RequirePackage{mathtools}
\RequirePackage{longtable}
\RequirePackage{multirow}

\def\le{\leqslant}

\def\<{\langle}
\def\>{\rangle}

\newcommand{\fkm}{\ensuremath{\mathfrak{m}}\xspace}
\newcommand{\fkn}{\ensuremath{\mathfrak{n}}\xspace}

\newcommand{\BC}{\ensuremath{\mathbb {C}}\xspace}
\newcommand{\BD}{\ensuremath{\mathbb {D}}\xspace}

\newcommand{{\BG}}{\ensuremath{\mathbb {G}}\xspace}

\newcommand{{\BK}}{\ensuremath{\mathbb {K}}\xspace}

\newcommand{\BR}{\ensuremath{\mathbb {R}}\xspace}

\newcommand{\BW}{\ensuremath{\mathbb {W}}\xspace}

\newcommand{\BZ}{\ensuremath{\mathbb {Z}}\xspace}

\newcommand{\CC}{\ensuremath{\mathcal {C}}\xspace}

\newcommand{\CF}{\ensuremath{\mathcal {F}}\xspace}
\newcommand{\CG}{\ensuremath{\mathcal {G}}\xspace}

\newcommand{\CN}{\ensuremath{\mathcal {N}}\xspace}

\newcommand{\CR}{\ensuremath{\mathcal {R}}\xspace}
\newcommand{\CS}{\ensuremath{\mathcal {S}}\xspace}

\newcommand{\CY}{\ensuremath{\mathcal {Y}}\xspace}

\newcommand{\GL}{\mathrm{GL}}

\DeclareMathOperator{\Hom}{Hom}

\newcommand{\id}{\ensuremath{\mathrm{id}}\xspace}

\newcommand{\ind}{{\mathrm{ind}}}

\DeclareMathOperator{\Jac}{Jac}

\DeclareMathOperator{\ord}{ord}

\newcommand{\SL}{{\mathrm{SL}}}

\DeclareMathOperator{\Sym}{Sym}

\newcommand{\ov}{\overline}

\newcommand{\lra}{\longrightarrow}

\newcommand{\supp}{\operatorname{supp}}
\newcommand{\Ext}{\mathrm{Ext}}
\newcommand{\Irr}{\mathrm{Irr}}
\newcommand{\Fin}{\mathrm{Fin}}
\newcommand{\Alg}{\mathrm{Alg}}
\newcommand{\St}{\mathrm{St}}
\newcommand{\sm}{\mathrm{sm}}
\newcommand{\Nrd}{\mathrm{Nrd}}
\newcommand{\Trd}{\mathrm{Trd}}
\newcommand{\soc}{\mathrm{soc}}
\newcommand{\length}{\mathrm{length}}
\newcommand{\MS}{\operatorname{Mult}}

\newtheorem{theorem}{Theorem}
\newtheorem{proposition}[theorem]{Proposition}
\newtheorem{lemma}[theorem]{Lemma}
\newtheorem {conjecture}[theorem]{Conjecture}
\newtheorem{corollary}[theorem]{Corollary}

\theoremstyle{definition}
\newtheorem{definition}[theorem]{Definition}

\newtheorem{remark}[theorem]{Remark}

\numberwithin{equation}{section}
\numberwithin{theorem}{section}

\setitemize[0]{leftmargin=*,itemsep=\the\smallskipamount}
\setenumerate[0]{leftmargin=*,itemsep=\the\smallskipamount}

\renewcommand{\to}{%
   \ifbool{@display}{\longrightarrow}{\rightarrow}%
   }
\let\shortmapsto\mapsto
\renewcommand{\mapsto}{%
   \ifbool{@display}{\longmapsto}{\shortmapsto}%
   }
\newlength{\olen}
\newlength{\ulen}
\newlength{\xlen}
\newcommand{\xra}[2][]{%
   \ifbool{@display}%
      {\settowidth{\olen}{$\overset{#2}{\longrightarrow}$}%
       \settowidth{\ulen}{$\underset{#1}{\longrightarrow}$}%
       \settowidth{\xlen}{$\xrightarrow[#1]{#2}$}%
       \ifdimgreater{\olen}{\xlen}%
          {\underset{#1}{\overset{#2}{\longrightarrow}}}%
          {\ifdimgreater{\ulen}{\xlen}%
             {\underset{#1}{\overset{#2}{\longrightarrow}}}
             {\xrightarrow[#1]{#2}}}}%
      {\xrightarrow[#1]{#2}}
   }
\makeatother
\newcommand{\xyra}[2][]{%
   \settowidth{\xlen}{$\xrightarrow[#1]{#2}$}%
   \ifbool{@display}%
      {\settowidth{\olen}{$\overset{#2}{\longrightarrow}$}%
       \settowidth{\ulen}{$\underset{#1}{\longrightarrow}$}%
       \ifdimgreater{\olen}{\xlen}%
          {\mathrel{\xymatrix@M=.12ex@C=3.2ex{\ar[r]^-{#2}_-{#1} &}}}%
          {\ifdimgreater{\ulen}{\xlen}%
             {\mathrel{\xymatrix@M=.12ex@C=3.2ex{\ar[r]^-{#2}_-{#1} &}}}
             {\mathrel{\xymatrix@M=.12ex@C=\the\xlen{\ar[r]^-{#2}_-{#1} &}}}}}%
      {\mathrel{\xymatrix@M=.12ex@C=\the\xlen{\ar[r]^-{#2}_-{#1} &}}}%
   }
\makeatletter
\newcommand{\xla}[2][]{%
   \ifbool{@display}%
      {\settowidth{\olen}{$\overset{#2}{\longleftarrow}$}%
       \settowidth{\ulen}{$\underset{#1}{\longleftarrow}$}%
       \settowidth{\xlen}{$\xleftarrow[#1]{#2}$}%
       \ifdimgreater{\olen}{\xlen}%
          {\underset{#1}{\overset{#2}{\longleftarrow}}}%
          {\ifdimgreater{\ulen}{\xlen}%
             {\underset{#1}{\overset{#2}{\longleftarrow}}}
             {\xleftarrow[#1]{#2}}}}%
      {\xleftarrow[#1]{#2}}
   }
\newcommand{\isoarrow}{%
   \ifbool{@display}{\overset{\sim}{\longrightarrow}}{\xrightarrow\sim}%
   }
   
\begin{document}

\title[Multiplicity One Property]{On Irreducibility and Multiplicity One Property for Non-Archimedean
Type-II Theta Lifting}
\author[Kaidi Wu]{Kaidi Wu}
\address{Department of Mathematics and New Cornerstone Science Laboratory, The University of Hong Kong}
\email{kaidiwu24@connect.hku.hk}

\thanks{}


\date{\today}
\begin{abstract}
We establish a multiplicity-one refinement of the non-archimedean type-II Howe correspondence: whenever the small theta lift is non-zero, it occurs in the big theta lift with Jordan–Hölder multiplicity one. Furthermore, we apply this result alongside the Godement–Jacquet zeta integral to compute big theta lift for Arthur-type representations. Based on these computations, we also propose a conjecture regarding the irreducibility of the big theta lift.
\end{abstract}

\maketitle
\tableofcontents

\section{Introduction}
Let $F$ be a non-Archimedean local field of characteristic zero and arbitrary residue characteristic, and let $\mathrm{D}$ be a central division algebra over $F$ of dimension $d^2$. Let $(G_n, G_m')$ be the type-II reductive dual pair, where $G_n \simeq \GL_n(\mathrm{D})$ and $G_m' \simeq \GL_m(\mathrm{D})$. There is a Weil representation $\omega_{n,m}$ of $G_n\times G_m'$ on $\CS(M_{n,m}(\mathrm{D}))$, the space of Bruhat-Schwartz functions on $n\times m$ matrices with coefficients on $\mathrm{D}$.

Given an irreducible representation $\pi$ of $G_n$, we would like to study the maximal $\pi$-isotypic quotient of $\omega_{n,m}$, namely $\omega_{n,m}/\CN_{\pi}$, where
\[
 \CN_{\pi}:=\bigcap_{\phi\in \Hom_{G_n}(\omega_{n,m},\pi)} \ker \phi.
\]
It is a representation of $G_n\times G_m'$ isomorphic to $\pi\boxtimes \Theta_{n,m}(\pi)$, where $\Theta_{n,m}(\pi)$ is a representation of $G_m'$ of finite length. We call $\Theta_{n,m}(\pi)$ the \textbf{big theta} of $\pi$. By definition, it is isomorphic to the coinvariant:
\[
 \Theta_{n,m}(\pi)\simeq (\pi^{\vee}\otimes\omega_{n,m})_{G_n},
\]
where $\pi^{\vee}$ is the contragredient representation of $\pi$. To better state our results, when $nm=0$, we set 
     \[
     \Theta_{n,m}(\pi)\simeq\begin{cases}
          0 &\text{ if } \pi \not\simeq \mathbf{1}_n;\\
           \mathbf{1}_m &\text{ if } \pi \simeq \mathbf{1}_n,
     \end{cases}
     \]
     where $\mathbf{1}_n$ is the trivial representation of $G_n$.
The most celebrated and significant result regarding the big theta, known as the Howe duality, is due to M\'inguez in~\cite[Theorem 1]{Ming08}.
\begin{theorem}[Howe duality]\label{Howe duality thm}
  Let $\pi$ be an irreducible representation of $G_n$.
\begin{enumerate}
\item If $\Hom_{G_n}(\omega_{n,m},\pi)\neq0$, then there exists a unique irreducible representation $\pi'$ of $G'_m$ such that
\[
\Hom_{G_n\times G'_m}(\omega_{n,m},\pi\otimes\pi')\neq0.
\]
Moreover,
\[
\dim\Hom_{G_n\times G'_m}(\omega_{n,m},\pi\otimes\pi')=1.
\]
\item Suppose $n\le m$. Then $\Hom_{G_n}(\omega_{n,m},\pi)\neq0$ and, if $\pi$ is the Langlands quotient of $\tau_1\times\cdots\times\tau_N$, where $\tau_1,\ldots,\tau_N$ are essentially square-integrable, then $\pi'$ is the Langlands quotient
\[
LQ(\nu_1^{-\frac{m-n-1}{2}}\times\cdots\times\nu_1^{\frac{m-n-1}{2}}
\times \tau_1^{\vee}\times\cdots\times\tau_N^{\vee}).
\]
\end{enumerate}
\end{theorem}
In Theorem~\ref{Howe duality thm}, $\times$ denotes normalized parabolic induction, and $LQ(\sigma_1\times \cdots\times \sigma_r)$ is the Langlands quotient of the standard representation $\sigma_1'\times \cdots\times\sigma_r'$, where $\sigma_1',\dots,\sigma_r'$ is a permutation of $\sigma_1,\dots,\sigma_r$. In addition, $\nu_n$ is a character of $G_n$ given by $|\Nrd(g)|_{F}^d$ for $g\in G_n$, where $\Nrd$ is the reduced norm.

In particular, part (1) shows that when $\Theta_{n,m}(\pi)\neq 0$, it has a unique irreducible quotient, and part (2) gives the Langlands parameter of this irreducible quotient. We call this irreducible quotient \textbf{small theta} and denote it by $\theta_{n,m}(\pi)$. 

Although M\'inguez's result explicitly determines the small theta $\theta_{n,m}(\pi)$, the big theta $\Theta_{n,m}(\pi)$ remains less understood. A basic and natural question is to determine the Jordan-H\"older multiplicity of small theta in big theta, i.e., the multiplicity of $\theta_{n,m}(\pi)$ occurs in the Jordan-H\"older factors of $\Theta_{n,m}(\pi)$, which we denote by $[\Theta_{n,m}(\pi):\theta_{n,m}(\pi)]$. The first main result of this article gives an answer to this question.
\begin{theorem}\label{multi one intro thm}
  Let $(n,m)$ be a pair of positive integers.  Let $\pi$ be an irreducible representation of $G_n$. If $\theta_{n,m}(\pi)\neq 0$, then 
    \[
    [\Theta_{n,m}(\pi):\theta_{n,m}(\pi)]=1.
    \]
\end{theorem}

The proof uses M\'inguez's dichotomy (Lemma~\ref{binary lem}).
Either the representation $\pi$ is off the
boundary of the rank filtration, or a suitable cuspidal factor can
be removed by taking a highest derivative. In the first case,
restriction to the open stratum gives an embedding of the
smooth dual of the big theta into a parabolically induced representation when $m\geq n$,
or into a big derivative with respect to a character when $m<n$.
Then it suffices to show that the parabolically induced representation or the big derivative is socle irreducible, which follows from the
results of Lapid--M\'inguez~\cite{LM16} and Chan~\cite{Ch24}.

In the second case, one chooses a cuspidal representation $\chi$ of
$G_l$ and an embedding
\[
 \pi\hookrightarrow\chi^{\times a}\times\rho,
 \qquad \rho\in\Irr(G_{n-la}),
\]
with $a$ maximal, so that $\rho$ is left $\chi$-reduced. Under the
noncriticality conditions
in the proof, Kudla's filtration gives a surjection
\begin{equation}\label{induction intro eq}
 (\chi^{\vee})^{\times a}\times
 \Theta_{n-la,m-la}(\rho)
 \twoheadrightarrow\Theta_{n,m}(\pi).
\end{equation}
Here $m\geq la$ whenever the right-hand side is nonzero. The precise
reduction is given in Proposition~\ref{big theta relative to highest derivative}.
The key additional input is that cuspidal reducedness is preserved
at the level of composition factors of the relevant theta lift
(Theorem~\ref{preserve reducedness thm}). This allows us to apply
Lemma~\ref{multiplicity one detector lem} and complete the proof by
induction on the size $nm$ of the dual pair.

It is worth mentioning that in recent joint work~\cite{GW26} with Zhibin Geng, we establish a similar multiplicity-one result for arbitrary Archimedean dual pairs. Our primary tool in that setting is the lowest degree $K$-type. While it is natural to ask whether this idea can be applied to the $p$-adic case, the drawbacks of the lattice model in even residue characteristic present significant technical difficulties. Because applications to number theory require a result that is independent of the residue characteristic, we return to the basic Jacquet module argument used here.

As we mentioned before, our ultimate goal is to determine big theta explicitly. In~\cite{CLLTZ25}, building on the earlier works of~\cite{APS17,FSX18}, a partial result was obtained by a homological argument together with the Godement-Jacquet zeta integral when $\mathrm{D}=F$.
\begin{theorem}[\cite{CLLTZ25}]\label{L-function holomorphic big theta thm}
   Let $m\geq n$ be positive integers.  Let $\pi$ be an irreducible smooth representation of $G_n$. We have 
\[
    \Theta_{n,m}\left(\pi\right) \simeq \begin{cases}
       \mathbf{1}_{m-n}\times \pi^\vee \quad & \textit{if $L\left(s,\pi\right)$ is holomorphic at $s_0=\frac{1+m-n}{2}$};\\[10pt]

        \pi^\vee\times \mathbf{1}_{m-n} \quad & \textit{if $L\left(s,\pi^\vee\right)$ is holomorphic at $s_0=\frac{1+m-n}{2}$}
    \end{cases}
\]
as finite length smooth representations. In particular, if both $L\left(s,\pi\right)$ and $L\left(s,\pi^\vee\right)$ are holomorphic at $s_0=\frac{1+m-n}{2}$, then $\Theta_{n,m}\left(\pi\right)$ is irreducible.
\end{theorem}

Motivated by this theorem, we introduce the following notation.
\begin{definition}
  Let $(n,m)$ be a pair of positive integers. An irreducible representation $\pi$ of $G_r$ is called \textbf{exceptional} (with respect to $(n,m)$) if both $L(s,\pi)$ and $L(s,\pi^{\vee})$ have a pole at $s_0:=\frac{d(m-n)+1}{2}$. Here, $r$ and $n$ are not required to be identical.
\end{definition}
Following Theorem~\ref{L-function holomorphic big theta thm}, it is natural to ask how to compute $\Theta_{n,m}(\pi)$ when $\pi$ is exceptional. The remainder of this article attempts to answer this question for Arthur-type representations using the multiplicity-one result of Theorem~\ref{multi one intro thm}.

To better illustrate our results, we first need to introduce some notation for the Arthur-type representation. For the reader's convenience, we assume $\mathrm{D}=F$ here, although our results hold for general $\mathrm{D}$. See subsection~\ref{Arthur-type sec} for more details about the case of inner forms.

An \textbf{Arthur-type representation} is the local component of an automorphic representation in the discrete spectrum (see Definition~\ref{Arthur-type rep def}). The building block of this family of representations is the Speh representation, which is the Langlands quotient
\[
U(\chi;a,b):= LQ(\nu^{\frac{b-1}{2}}\St(a,\chi)\times \cdots\times \nu^{\frac{1-b}{2}}\St(a,\chi)),
\]
where $a,b$ are positive integers and $\chi$ is a unitary cuspidal representation of $G_l$. Here, $\St(a,\chi)$ is the unique irreducible quotient of 
\[
 \nu^{\frac{1-a}{2}}\chi\times \cdots\times \nu^{\frac{a-1}{2}}\chi.
\]
Thus, $U(\chi;a,b)$ is a representation of $G_{abl}$. Let $\pi=U(\chi;a,b)$. Since $\pi^{\vee}\simeq U(\chi^{\vee};a,b)$, we have $L(s,\pi)$ is holomorphic at $s$ if and only if $L(s,\pi^{\vee})$ is, see Lemma~\ref{L-function holo lem} for further explanation.

In general, an Arthur-type representation of $G_n$ is parameterized by an $A$-parameter, which is a continuous homomorphism
\[
\phi: W_F\times \SL_2(\BC)\times\SL_2(\BC)\lra \GL_n(\BC)
\]
such that the restriction of $\phi$ to the Weil group $W_F$ has a bounded image, and the restriction of $\phi$ to two $\SL_2(\BC)$ factors is algebraic. Let $\tau_{\chi}$ be the irreducible representation of $W_F$ that corresponds to $\chi$ under local Langlands correspondence~\cite{HT01}, and let $\Sym^{k}(\BC^2)$ be the $k$-th symmetric power of the standard representation of $\SL_2(\BC)$. The $A$-parameter of the Speh representation $U(\chi;a,b)$ is
\[
 \tau_{\chi}\boxtimes \Sym^{a-1}(\BC^2)\boxtimes\Sym^{b-1}(\BC^2),
\]
which is also irreducible. 

Let $\pi$ be an Arthur-type representation of $G_n$, whose $A$-parameter has an irreducible decomposition:
\[
\phi_{\pi}=\bigoplus_{i=1}^k \phi_i.
\]
Let $U(\chi_i;a_i,b_i)$ be the Speh representation whose $A$-parameter is $\phi_i$. Then
\[
\pi\simeq U(\chi_1;a_1,b_1)\times \cdots\times U(\chi_k;a_k,b_k).
\]
It is a fact that any product of Speh representations is still irreducible. Let $\pi_{\mathrm{h}}$ be the product of Speh block $\pi_i=U(\chi_i;a_i,b_i)$ such that $L(s,\pi_i)$ is holomorphic at $s_0=\frac{m-n+1}{2}$. Furthermore, let $\pi_{\mathrm{e}}$ denote the product of the Speh blocks that are exceptional with respect to $(n,m)$. Then
\[
\pi\simeq \pi_{\mathrm{h}}\times \pi_{\mathrm{e}}.
\]
The first step towards computing the big theta lift is that we can separate $\pi_{\mathrm{h}}$ from the big theta lift.
\begin{proposition}\label{separation prop intro}
   Let $n,m$ be two positive integers. Let $\pi$ be an Arthur-type representation of $G_n$ such that $\pi_{\mathrm{e}}\in\Irr(G_l)$. Then 
    \begin{enumerate}
        \item If $m<n-l$, then $\Theta_{n,m}(\pi)=0$;
        \item If $m\geq n-l$, then
    \end{enumerate}
    \[
    \Theta_{n,m}(\pi)\simeq \pi_{\mathrm{h}}^{\vee}\times \Theta_{l,m-n+l}(\pi_{\mathrm{e}}).
    \]
\end{proposition}

For more general statements regarding extension groups, readers are referred to Theorem~\ref{separation thm}. We remark that $\pi_e$ is also exceptional with respect to $(l,m-n+l)$. 

According to this proposition, it suffices to determine the big theta of a product of exceptional Speh blocks. We complete the computation of a single Speh block in this article.
\begin{theorem}\label{big theta speh thm intro}
  Let $n,m$ be two positive integers. Let $\pi=U(\chi;a,b)$ be an exceptional Speh representation of $G_n$ with respect to $(n,m)$. Then
    \begin{equation}
        \Theta_{n,m}(\pi)\simeq\begin{cases}
            0 &\text{if } a>1,m<n  ;\\
            \mathbf{1}_{m-n}\times \pi &\text{if } a>1,m\geq n;\\
       \nu^{\frac{2n-m}{4}}\mathbf{1}_{\frac{m}{2}}\times  \nu^{\frac{m-2n}{4}}\mathbf{1}_{\frac{m}{2}} & \text{if } a=1 \text{ and } m \text{ is even}.
        \end{cases}
    \end{equation}
\end{theorem}
For the statement of inner forms, see Remark~\ref{speh theta rem}.

We remark that when $a=1$ and $m$ is odd, $\pi$ is never exceptional. Moreover, we observe that $\nu^{\frac{2n-m}{4}}\mathbf{1}_{\frac{m}{2}}\times  \nu^{\frac{m-2n}{4}}\mathbf{1}_{\frac{m}{2}}$ is reducible when $m\geq n$. The proof of this theorem starts from computing the Hom-space associated to each graded piece in the rank filtration. Then we use the multiplicity-one Theorem~\ref{multi one intro thm} and the Godement-Jacquet zeta integral to glue them together. Crucially, the rank filtration on the Fourier dual side provides additional information that is essential to our argument.

This theorem, together with Proposition~\ref{separation prop intro}, completely determines $\Theta_{n,m}(\pi)$ when $\pi$ is an Arthur-type representation such that
$$\ord_{s=s_0} L(s,\pi)\leq 1.$$
The case where $\ord_{s=s_0} L(s,\pi)> 1$ appears to be much more complicated. But we can still say something about the irreducibility of big theta. We propose the following conjecture.
\begin{conjecture}\label{irr conj intro}
  Let $m\geq n$ be positive integers.  Let $\pi$ be an Arthur-type representation of $G_n$ such that
    \[
   \pi\simeq u(\chi_1;a_1,b_1)\times \cdots\times u(\chi_k;a_k,b_k),
    \]
    where block $u(\chi_i;a_i,b_i)$ is exceptional with respect to $(n,m)$ for each $1\leq i\leq k$. Then $\Theta_{n,m}(\pi)$ is irreducible if and only if $a_i>1$ for each $1\leq i\leq k$.
\end{conjecture}

Here, $u(\chi;a,b)$ replaces $U(\chi;a,b)$ as the building block for Arthur-type representations of inner forms (see subsection~\ref{Arthur-type sec} for definition). As evidence for this conjecture, we prove it in the following special cases.
\begin{proposition}
 Retain the notation and assumption in Conjecture~\ref{irr conj intro}.
    \begin{enumerate}
    \item Assume $n=m$ and $\mathrm{D}=F$. If $a_i=1$ for some $i$, then $\Theta_{n,n}(\pi)$ is reducible. 
    \item Assume $m \geq n$. If $\pi$ is a repeated product of exceptional blocks $u(\chi;a,b)$ with $a>1$, i.e. $\chi_i=\chi$, $a_i=a$ and $b_i=b$ for any $1\leq i\leq k$, then $\Theta_{n,m}(\pi)$ is irreducible. 
\end{enumerate}
\end{proposition}
Also see Proposition~\ref{reducible prop} for part (1) and Proposition~\ref{irreducible prop} for part (2). 

To prove (1), we borrow some information from (non-tempered) GGP conjecture. In what follows, we assume $\mathrm{D}=F$.
\begin{definition}
    Let $n,m$ be positive integers and $\pi\in \Irr(G_n)$. Then $\lambda\in \Irr(G_{m-1})$ is called a \textbf{(GGP-)test} representation for $\pi$ if
    \[
    \Hom_{G_{m-1}}(\Theta_{n,m}(\pi),\lambda)\neq 0 \text{ but } \Hom_{G_{m-1}}(\theta_{n,m}(\pi),\lambda)=0.
    \]
\end{definition}
It is clear that if a test representation for $\pi$ exists, then $\Theta_{n,m}(\pi)$ is non-zero and reducible. We emphasize that when $\pi$ and $\theta_{m-1,n}(\lambda)$ are Arthur-type representations, these two Hom-spaces are computable. The Hom-space for the small theta reduces to a branching problem from $G_m$ to $G_{m-1}$, which is now well understood thanks to a series of works by Chan and Chan–Pattanayak~\cite{Cha_qbl,Cha_csq,Cha_csq_ii,Cha_csq_iii,Cha_duality,CP25}. For the Hom-space of big theta, the see-saw diagram 
 \begin{equation}\label{see saw diagram}
   \begin{tikzpicture}
  \matrix[matrix of math nodes,row sep=2em,column sep=4em,minimum width=2em] (m) {
     G_m & G_n\times G_n\\
    G_{m-1}\times G_1   & G_n,\\};
  \draw[-] (m-1-1) -- (m-2-1); 
  \draw[-] (m-1-2) -- (m-2-2); 
  \draw[-] (m-1-1) -- (m-2-2); 
  \draw[-] (m-1-2) -- (m-2-1); 
\end{tikzpicture} 
\end{equation}
implies the isomorphism
\[
 \Hom_{G_{m-1}}(\Theta_{n,m}(\pi),\lambda)\simeq \Hom_{G_n}(\Theta_{m-1,n}(\lambda)\otimes \Omega_n , \pi).
\]
Here, $\Omega_n$ is a representation of $G_n$ on $\CS(F^n)$ given by 
\[
( g\cdot f)(x)= |\det g|_F^{-1/2}f(g^{-1}x).
\]
In particular, if the A-parameter of $\theta_{m-1,n}(\lambda)$ and $\pi$ is relevant (see Definition~\ref{relevant def}), then the non-tempered GGP-conjecture for this Fourier-Jacobi model (which is now a theorem due to Chan~\cite{Ch22}) ensures that this Hom-space is non-zero.

\section*{Acknowledgments}
The author is supported by the National Natural Science Foundation of China (Project No. 123B1004). He is also partially supported by the New Cornerstone Science Foundation through the New Cornerstone Investigator Program awarded to Professor Xuhua He.

This project grew out of joint work with Zhibin Geng, which establishes a similar multiplicity one result for arbitrary archimedean dual pairs. The author thanks Geng for his support and helpful discussions. He is also grateful to Kei Yuen Chan, Basudev Pattanayak, and Mohammed Saad Qadri for valuable conversations regarding the material in the appendix. The author would like to thank his advisor Xuhua He for continued support and encouragement.

During the drafting process, AI tools were used to search for references. After the initial draft was completed, AI (specifically ChatGPT) was used for proofreading, language revision, and the creation of Figure~\ref{fig:segments-exceptional-speh}. No AI tools were used in any other capacity. 
\section{Preliminary}
\subsection{Notations}
Let $F$ be a non-Archimedean local field of characteristic zero and residual characteristic $p>0$, and let $\mathrm{D}$ be a central division algebra over $F$ of dimension $d^2$. We fix a uniformizer $\varpi$ of $F$ and denote by $q$ the cardinality of its residue field. We write $|\cdot|_F$, or simply $|\cdot|$ when no confusion can arise, for the normalized absolute value on $F$. Throughout the article, an additive character $\psi:F\lra \BC^{\times}$ is fixed.

\subsubsection{Notations for dual pairs}
Let $n,m$ be positive integers. We denote by $M_{n,m}(\mathrm{D})$ the space of $n\times m$ matrices with entries in $\mathrm{D}$, and by
\[
\Nrd:M_{n,n}(\mathrm{D})\longrightarrow F
\]
the reduced norm. The group $\GL_n(\mathrm{D})$ of invertible matrices in $M_{n,n}(\mathrm{D})$ will be denoted by $G_n$. When considering a dual pair, we use $G_m'$ to denote the second member, in order to distinguish it from the first member.

Let $\nu_n$ denote the character of $G_n$ given by
$$\nu_n(g) := \vert{}\Nrd(g)\vert{}_F^d \quad \text{for } g \in G_n.$$

Denote by $S_{n,m}=\CS(M_{n,m}(\mathrm{D}))$ the vector space of locally constant compactly supported complex-valued functions. We use the Schr\"odinger model of the metaplectic representation $\omega_{n,m}$ restricted to the dual pair $G_n\times G_m'$:
\begin{equation}
\omega_{n,m}(g,g')
=\nu_n(g)^{-m/2}\,\sigma_{n,m}(g,g')\,\nu'_m(g')^{n/2},
\end{equation}
where $\sigma_{n,m}$ is defined by
\[
(\sigma_{n,m}(g,g')\Phi)(x)=\Phi(g^{-1}xg')
\]
for $g\in G_n$, $g'\in G_m'$, $x\in M_{n,m}(\mathrm{D})$, and $\Phi\in S_{n,m}$.

To effectively apply the results of~\cite{Ming08}, we follow his convention and define
\[
\ov{\Theta}_{n,m}(\pi) := \Hom_{G_n}(\sigma_{n,m}, \pi)_{\sm} \text{ for }\pi\in \Irr(G_n).
\]
For later convenience, we would like to extend the theta lift to the situation when $nm=0$.  For $\pi\in \Irr(G_n)$, we set 
     \[
     \ov{\Theta}_{n,m}(\sigma)\simeq\begin{cases}
          0 &\text{ if } \sigma \not\simeq \mathbf{1}_n;\\
           \mathbf{1}_m &\text{ if } \sigma \simeq \mathbf{1}_n.
     \end{cases}
     \]
Note that $\ov{\Theta}(\pi)\nu'^{-n/2}= \Theta(\nu^{-m/2}\pi)^{\vee}$.

Throughout the paper, we use a \textbf{prime} to distinguish the notation associated with the second member $G_m'$ of a dual pair from the corresponding notation associated with $G_n$. For example, characters and subgroups associated with $G_m'$ will be denoted by primed symbols whenever necessary.

\subsubsection{Notations for representations}
Throughout this article, we consider only smooth complex representations, and the word ``representation'' will always mean smooth complex representation. The category of these representations of some $\ell$-group $G$ is denoted by $\Alg(G)$. For a possibly non-smooth representation $V$, we denote by $V_{\sm}$ the subrepresentation consisting of its smooth vectors.  

We denote by $\Irr(G_n)$ the set of equivalence classes of irreducible representations of $G_n$. Let $\Irr$ denote the disjoint union
\[
\Irr=\bigcup_{n\geq 0}\Irr(G_n),
\]
and let $\CC$ be the subset of $\Irr$ formed by cuspidal representations. If $\pi\in\Irr(G_n)$, then $\mathrm{gr}(\pi):=n$.

For a finite length representation $\pi$ of $G_n$, let $\soc(\pi)$ be the socle (i.e. maximal semisimple submodule) of $\pi$ and let $\cos(\pi)$ be the cosocle (i.e. maximal semisimple quotient) of $\pi$. $\pi$ is called \textbf{socle irreducible} if $\soc(\pi)$ is irreducible and 
\[
[\pi: \soc(\pi)]=1.
\]

Let $k$ be an integer such that $0\leq k\leq n$. Let $P_{k,n-k}$ be the standard parabolic subgroup of $G_n$ consisting of block upper triangular matrices whose standard Levi subgroup is $L_{k,n-k}\simeq G_{k}\times G_{n-k}$. The opposite parabolic subgroup $\ov{P}_{k,n-k}$ is defined as the subgroup consisting of transpose matrices in $P_{k,n-k}$. When the ambient group $G_n$ is clear, we will simply write $P_k$ instead of $P_{k,n-k}$. When $k=0$ or $k=n$, $P_{k,n-k}=\ov{P}_{k,n-k}=G_n$. We also set $U_n$ to be the subgroup of $G_n$ consisting of strictly upper triangular matrices.

Let $\sigma$ be a representation of $L_{k,n-k}$. We use 
\[
\ind_{P_{k,n-k}}^{G_n}(\sigma)
\]
to denote the \textbf{normalized} parabolic induction of $\sigma$ to $G_n$. When $\sigma=\sigma_1\boxtimes\sigma_2$, where $\sigma_1$ is a representation of $G_k$ and $\sigma_2$ is a representation of $G_{n-k}$, we use the shorthand
\begin{equation}\label{times parabolic induction eq}
    \sigma_1\times \sigma_2:= \ind_{P_{k,n-k}}^{G_n}(\sigma).
\end{equation}
We extend this notation to $\sigma_1\times \cdots\times \sigma_t$ by associativity. When $\sigma_1=\cdots=\sigma_t=\chi$, it is also denoted by $\chi^{\times t}$.

Let $\pi$ be a representation of $G_n$, we will use
\[
r^{G_n}_{k,n-k}(\pi)
\]
to denote the \textbf{normalized} Jacquet module of $\pi$ with respect to $P_{k,n-k}$. Similarly, 
\[
\ov{r}^{G_n}_{k,n-k}(\pi)
\]
is for the normalized Jacquet module of $\pi$ with respect to $\ov{P}_{k,n-k}$. When $G_n$ is clear from the context, we will simply write $r_{k,n-k}$ and $\ov{r}_{k,n-k}$.

We will freely use the following standard identities:
\begin{enumerate}
    \item $\ind_{P_{k,n-k}}^{G_n}(\sigma_1\boxtimes \sigma_2)\simeq \ind_{\ov{P}_{n-k,k}}^{G_n}(\sigma_2\boxtimes \sigma_1)$.
    \item $(\sigma_1\times\sigma_2)^{\vee}\simeq \sigma_1^{\vee}\times \sigma_2^{\vee}$.
    \item Second adjointness: 
    \[
    \Hom_{G_n}(\ind_{P_{k,n-k}}^{G_n}(\sigma),\pi)\simeq \Hom_{L_{k,n-k}}(\sigma,\ov{r}_{k,n-k}(\pi)).
    \]
\end{enumerate}

There is a representation appearing in the filtration of the Weil representation. Let $\rho_n:=\CS(G_n)$ be the space of Bruhat-Schwartz functions on $G_n$. It is a representation of $G_n\times G_n'$ by
\begin{equation}\label{rho def eq}
   ( \rho_n(g,g')\Phi)(h):= \Phi(g^{-1} hg'), g\in G_n, g'\in G_n'.
\end{equation}

\subsubsection{Notations for segments}
Let $\rho\in \CC$. Then there exists a unique positive integer $s_{\rho}$ such that $\rho\times |\Nrd|^{\pm s_{\rho}}\rho$ is reducible. Let $\nu_{\rho}:= |\Nrd|^{ s_{\rho}}$. For $x,y\in \BR$ with $y-x\in \BZ_{\geq 0}$, we define the segment $[x,y]_{\rho}$ to be the set $\{\nu_{\rho}^x\rho,\cdots,\nu_{\rho}^y\rho\}$. And we call this segment based on $\rho$. If $\rho$ is the trivial representation of $G_1$, we will simply write $[x,y]$. 

We introduce some useful operations on a segment $\Delta=[x,y]_{\rho}$. Define
\begin{enumerate}
    \item $\Delta^-:= [x,y-1]_{\rho}$ and ${}^-\Delta:=[x+1,y]_{\rho}$;
    \item $\ell(\Delta):= y-x+1$ and $\Delta^{\vee}:=[-y,-x]_{\rho^{\vee}}$;
    \item $b(\Delta):=x$ and $e(\Delta):=y$, where ``b'' is for beginning and ``e'' is for ending.
\end{enumerate}
By convention, we set $[x,y]=\emptyset$ whenever $y<x$. Let $\Delta$ and $\Delta'$ be two segments. We say $\Delta$ and $\Delta'$ are linked if $\Delta\cup \Delta'$ is still a segment and $\Delta\not\subset \Delta', \Delta'\not\subset \Delta$. The segment $\Delta$ is strictly contained in $\Delta'$ if $\Delta\subset \Delta'$ and 
\[
b(\Delta')<b(\Delta)\quad ,\quad e(\Delta)<e(\Delta').
\]
We say that $\Delta$ proceeds $\Delta'$, which is denoted by $\Delta\prec \Delta'$, if 
\begin{enumerate}
    \item $\Delta$ and $\Delta'$ are linked, and
    \item $b(\Delta)<b(\Delta')$.
\end{enumerate}

 A multisegment is a multiset of segments. Let $\MS$ be the set of multisegments. Given $\fkm,\fkn\in \MS$, we write $\fkm\not\prec \fkn$ if $\Delta\not\prec \Delta'$ for any $\Delta\in\fkm$ and $\Delta'\in\fkn$. For a multisegment $\fkm:=\Delta_1+\cdots+\Delta_k$, we define 
\[
\fkm^{\vee}:=\Delta_1^{\vee}+\cdots +\Delta^{\vee}_k \text{ and }\fkm^{-}:=\Delta_1^{-}+\cdots +\Delta^{-}_k.
\]

\subsection{Cuspidal supports and blocks decomposition}
Let $M(\CC)$ be the set of multisets of cuspidal representations, or equivalently, functions
\[
\Omega: \CC \lra \BZ_{\geq 0}
\]
with finite support. If $\Omega\in M(\CC)$, set 
\[
|\Omega|:= \sum_{\chi\in \CC} \Omega(\chi)\mathrm{gr}(\chi).
\]
Let $\pi\in \Irr(G_n)$. By~\cite[III.2.1]{Ber92}, there exists a standard parabolic subgroup $P$ and a cuspidal representation $\chi$ of the standard Levi subgroup $L$ such that $\pi\hookrightarrow \ind_P^{G_n}(\chi)$. Since $L$ is the product of some copies of general linear groups, $\chi$ gives an element $\Omega\in M(\CC)$ such that $|\Omega|=n$. Moreover, \cite[III.2.1, Theorem 18]{Ber92} confirms that $\Omega$ is independent of the choice of $(L,\chi)$, which is the cuspidal support of $\pi$ and is denoted by $\supp(\pi)$.

 Let $\Fin(G_n)$ be the full subcategory of $\Alg(G_n)$ consisting of finite-length representations. For $\Omega\in M(\CC)$, we define $\Fin(\Omega)$ to be the full subcategory of $\Fin(G_n)$ such that, for every irreducible subquotient $\pi$ of every object of $\Fin(\Omega)$, one has $\supp(\pi)=\Omega$.

By Bernstein decomposition Theorem (cf.~\cite[III.2.2]{Ber92}), the category $\Fin(G_n)$ is the direct sum of the blocks $\Fin(\Omega)$, namely,
\begin{equation}\label{Bernstein decomp eq}
  \Fin(G_n)=\bigoplus_{\substack{\Omega\in M(\CC)\\ |\Omega|=n}} \Fin(\Omega).  
\end{equation}
We recall some basic results about the representation theory of $G_n=\GL_n(\mathrm{D})$. 
\begin{lemma}\label{cuspidal product irr lem}
    Let $\chi$ be a cuspidal representation of $G_n$. Then 
    \[
    \underbrace{\chi\times \cdots \times \chi}_{k}
    \]
    is irreducible for any positive integer $k$.
\end{lemma}
\begin{proof}
    $\rho$ can be regarded as a singleton segment, and it is not linked to itself. The result follows from~\cite[Theorem 4.2]{Ming09}. Also see~\cite[Lemma A.5]{LM16}.
\end{proof}
We will simply use $k[\chi]$ to denote the $\Omega\in M(\CC)$ consisting of $k$-copies $\chi$. The above lemma shows that the subcategory $\Fin(k[\chi])$ has a unique irreducible representation, that is $\chi^{\times k}$.

\subsection{Zelevinsky classification and Bernstein-Zelevinsky filtration}
The Zelevinsky classification parametrizes the irreducible representations of $G_n$ by multisegments, and is dual to the Langlands classification via the Aubert–Zelevinsky involution. It was completed by Zelevinsky in~\cite{Ze80} for the case $\mathrm{D}=F$, and further extended to general $\mathrm{D}$ by M\'inguez and S\'echerre in~\cite{MS13}.

For a segment $\Delta=[x,y]_{\chi}$, let $Z(\Delta)$ (resp. $\St(\Delta)$) be the unique irreducible submodule (resp. quotient) of
\[
 \nu_{\chi}^x\chi\times \cdots\times \nu_{\chi}^y \chi.
\]
They are defined as trivial representations of $G_0$ when $\Delta=\emptyset$. For any multisegment $\fkm$, we can write it as $\fkm=\Delta_1+\cdots+\Delta_k$ such that
\[
 \Delta_i \not\prec \Delta_j \text{ whenever } i<j.
\]
Then
\[
 Z(\Delta_1)\times \cdots\times Z(\Delta_k)
\]
has a unique irreducible submodule, which is independent of the possible labeling of $\Delta_i$ and denoted by $Z(\fkm)$.  
\begin{theorem}
    The map $\fkm\mapsto Z(\fkm)$ is a bijection from $\MS$ to $\Irr$.
\end{theorem}
The theorem for $\mathrm{D}=F$ is proved by a decreasing filtration due to Bernstein-Zelevinsky~\cite{BZ77}. Because it is crucial for constructing the test representation in section~\ref{reducible sec}, we shall recall some relevant results here. 

Let $H$ be a closed subgroup of a $\ell$-group $G$, and let $\sigma$ be a smooth representation of $H$. We use $c-\ind_H^G\sigma$ to denote the normalized compact induction of $\sigma$ to $G$. Namely, it consists of sections that are locally constant and compactly supported on $H\backslash G$.

Let $M_n$ be the mirabolic subgroup of $G_n$ consisting of matrices with the last row $(0, \ldots, 0,1)$. We shall also regard $G_{n-1}$ as a subgroup of $M_n$ via the embedding $g \mapsto \begin{pmatrix} g & \\ & 1 \end{pmatrix}$.  Let $V=V_{n}$ be the unipotent radical of $M_n$. We define $\widetilde{\psi}: V \rightarrow \mathbb{C}^{\times}$ by $\widetilde{\psi}(v)=\psi(v_{n-1})$, where $v_{n-1}$ is the last entry in $v$. Note that the action of $M_{n-1}$ stabilizes $\widetilde{\psi}$. 

For a character $\lambda$ of $V$ and a representation $\pi$ of $M_n$, define 
\[  \pi_{V, \lambda} = \delta^{-1/2} \pi /\langle \pi(v)\cdot x-\lambda(v)x : v \in V, x\in \pi \rangle ,
\]
where $\delta$ is the modular character of $M_n$. When $\lambda$ is trivial (resp. $\lambda=\widetilde{\psi}$), we regard $\pi_{V,\lambda}$ as a $G_{n-1}$-representation (resp. $M_{n-1}$-representation).

Define 
\[  \Phi^+: \mathrm{Alg}(M_{n-1}) \rightarrow \mathrm{Alg}(M_{n}) ; \quad \Psi^+:\mathrm{Alg}(G_{n-1}) \rightarrow \mathrm{Alg}(M_{n}) 
\]
\[  \Phi^-: \mathrm{Alg}(M_n)\rightarrow \mathrm{Alg}(M_{n-1}); \quad \Psi^-: \mathrm{Alg}(M_n) \rightarrow \mathrm{Alg}(G_{n-1}) .
\]
by
\[  \Phi^+(\pi)=c-\mathrm{ind}_{M_{n-1}V_n}^{M_{n}} \pi\boxtimes \widetilde{\psi}, \quad \Psi^+(\pi)=c-\mathrm{ind}_{G_{n-1}V_n}^{M_{n}} \pi \boxtimes \BC_{\mathrm{triv}} ,
\]
\[  \Phi^-(\pi)=\pi_{V_n, \widetilde{\psi}}, \quad \Psi^-(\pi)=\pi_{V_n, \mathrm{triv}} ,
\]
where ``$\mathrm{triv}$'' is the trivial character of $V$.

For $\pi \in \mathrm{Alg}(G_n)$, the $k$-th BZ-derivatives of $\pi$ is defined as:
\[  \pi^{(k)} =\Psi^-(\Phi^-)^{k-1}(\pi|_{M_n}) \in \Alg(G_{n-k})
\]
and the $k$-th shifted BZ-derivatives of $\pi$ is defined as:
\[  \pi^{[k]}=\nu^{1/2}\cdot \pi^{(k)} .
\]

Let $d$ be the largest integer such that $\pi^{(d)} \neq 0$. We shall call $d$ to be the \textbf{depth} of $\pi$. We also set the highest (shifted) BZ-derivative $\pi^-:=\pi^{[d]}$. The Bernstein-Zelevinsky filtration implies that $\pi|_{G_{n-1}}$ admits a filtration:
\[ 0= \pi_{n} \subset \pi_{n-1} \subset \ldots \pi_1 \subset \pi_0 = \pi|_{G_{n-1}} 
\]
such that
\[   \pi_{k-1}/\pi_{k} \cong  \pi^{[k]} \times c-\mathrm{ind}_{U_{k-1}}^{G_{k-1}} \psi_{k-1} ,
\]
where $\psi_{k-1}$ is a non-degenerate character on $U_{k-1}$.

We will also need the following Proposition about the Zelevinsky classification.
\begin{proposition}\label{Zele class_properties}
    Let $\mathfrak{m},\mathfrak{n}\in\MS$. Then
\begin{enumerate}
\item \label{part: non-preceding socle}
If $\mathfrak{m}\not\prec\mathfrak{n}$ then $Z(\mathfrak{m}+\mathfrak{n})=\soc(Z(\mathfrak{m}) \times Z(\mathfrak{n}))=\cos(Z(\mathfrak{n})\times Z(\mathfrak{m}))$.
In particular, if both $\mathfrak{m}\not\prec\mathfrak{n}$ and $\mathfrak{n}\not\prec\mathfrak{m}$, i.e. if no $\Delta \in \mathfrak{m}$ and $\Delta' \in \mathfrak{n}$ are linked,
then $Z(\mathfrak{m}) \times Z(\mathfrak{n})$ is irreducible.

\item \label{part: contragredient}
$Z(\mathfrak{m}^\vee)=Z(\mathfrak{m})^\vee$.

\item \label{part: highest derivative}
$Z(\fkm)^-=Z(\fkm^-)\cdot \nu^{1/2}$
\end{enumerate}
\end{proposition}
\begin{proof}
    For part (1), (2), see~\cite[Proposition 3.5]{LM16}. Part (3) is a Theorem in~\cite[8.1 Theorem]{Ze80}.
\end{proof}

\subsection{cuspidal reducedness and left derivatives}
In this subsection, we introduce one of the main tools to establish the multiplicity one property, that is, the derivatives which allow us to reduce the problem to smaller dual pairs.
\begin{definition}
    Let $\pi\in \Irr(G_n)$ and let $\chi$ be a cuspidal representation of $G_l$. We define $a_{\chi}(\pi)$ to be the maximal non-negative integer $a$ which there exists some $\rho\in \Irr(G_{n-al})$ such that
    \[
    \pi \hookrightarrow \chi^{\times a} \times \rho.
    \]
    If $a_{\chi}(\pi)=0$, $\pi$ is called \textbf{(left) $\chi$-reduced}.
\end{definition}
Let $\pi\in\Fin(G_n)$ and let $\rho\in \Irr(G_k)$. Assume $k\leq n$. Then by Bernstein decomposition~\eqref{Bernstein decomp eq}, there exist finite subsets $\{\Lambda_i\}$ and $\{\Lambda'_j\}$ of $M(\CC)$, such that
\[
r_{k,n-k}(\pi)= \bigoplus_{i}\bigoplus_j \Pi_{i,j},
\]
where each irreducible component $\pi_i\otimes\pi_j'$ in $\Pi_{i,j}$ has the property that $\supp(\pi_i)=\Lambda_i$ and $\supp(\pi_j')=\Lambda_j'$. Define
\begin{equation}\label{Jac decomposition eq}
    \Jac_{\rho}(\pi):= \bigoplus_j \Pi_{i,j}, \text{ where } \Lambda_i=\supp(\rho).
\end{equation}
\begin{lemma}\label{reducedness characterization lem}
    Let $\pi\in \Irr(G_n)$, and let $\chi$ be a cuspidal representation of $G_l$. Then $\pi$ is left $\chi$-reduced if and only if $\Jac_{\chi}(\pi)=0$.  
\end{lemma}
\begin{proof}
If $\pi$ is not left $\chi$-reduced, that is, $\pi\hookrightarrow \chi\times \rho$ for some $\rho\in \Irr(G_{n-l})$, then by Frobenius reciprocity, $\rho$ is an irreducible component in $\Jac_{\chi}(\pi)$.

Suppose $\Jac_{\chi}(\pi)\neq 0$, equivalently, there exists some $j$ such that in the decomposition~\eqref{Jac decomposition eq}, $\Pi_{i,j}\neq 0$. Then we take an irreducible quotient of $\Pi_{i,j}$. It is isomorphic to $\chi\boxtimes \rho$ for some $\rho\in \Irr(G_{n-l})$ since the only irreducible object in $\Fin([\chi])$ is $\chi$. Thus,  
\[
\pi \twoheadrightarrow r_{n,n-l}(\pi) \twoheadrightarrow \chi \boxtimes \rho
\]
induces a non-zero map $\pi\lra \chi\times \rho$, which is injective since $\pi$ is irreducible.
\end{proof}

To make sure that the left highest derivative of some cuspidal representation $\chi$ is well-defined, we need the following lemma.
\begin{lemma}\label{Jac of left chi seq}
 Let $\chi$ be a cuspidal representation of $G_l$.  If $\rho\in \Irr$ is left $\chi$-reduced, then $\Jac_{\chi^{\times a}}(\chi^{\times a} \times \rho)=\chi^{\times a} \times\rho$.
\end{lemma}
\begin{proof}
    Let $n=\mathrm{gr}(\rho)$. By geometric Lemma~\cite[III.1.2 Lemma 24]{Ber92}, we have
    \[
    r_{i,al-i}(\chi^{\times a}) \neq 0
    \]
    only if $i=kl$ for some $k\leq a$. Moreover, any irreducible subquotient of $r_{kl,(a-k)l}(\chi^{\times a})$
    is $\chi^{\times k} \boxtimes \chi^{\times (a-k)}$. Hence, by geometric Lemma again, the irreducible component in $r_{al, n}(\chi^{\times a} \times \rho)$ is 
    \[
    (\chi^{\times k}\times \alpha )\boxtimes (\chi^{\times (a-k)} \times \beta),
    \]
    where $\alpha\boxtimes\beta$ is some irreducible component of $r_{(a-k)l, n-(a-k)l}(\rho)$. Since $\rho$ is left $\chi$-reduced, $\alpha\neq \chi^{\times (a-k)}$ when $k<a$. Therefore,
    the irreducible component in $r_{al, n}(\chi^{\times a} \times \rho)$ such that the first factor has cuspidal support $a[\chi]$ is $\chi^{\times a}\boxtimes \rho$, which has multiplicity one. 
\end{proof}

\begin{corollary}
     Let $\pi\in \Irr(G_n)$, and let $\chi$ be a cuspidal representation of $G_l$. If $a_{\chi}(\pi)=a$, then there exists a unique $\rho\in \Irr(G_{n-al})$ such that $\pi\hookrightarrow \chi^{\times a} \times \rho$.
\end{corollary}
\begin{proof}
    Suppose that there exists $\rho'\neq \rho$ such that $\pi\hookrightarrow \chi^{\times a} \times \rho'$. Then, by Frobenius reciprocity, $\chi^{\times a} \boxtimes\rho$ and $\chi^{\times a} \boxtimes\rho'$ are irreducible components of $\Jac_{\chi^{\times a}}(\pi)$. On the other hand, since Jacquet functor is exact,
    \[
    \Jac_{\chi^{\times a}}(\pi) \hookrightarrow \Jac_{\chi^{\times a}} (\chi^{\times a} \times \rho)=\chi^{\times a} \boxtimes\rho,
    \]
    which leads to a contradiction.
\end{proof}

\begin{definition}
    Let $\pi\in \Irr(G_n)$, and let $\chi$ be a cuspidal representation of $G_l$. Assume $a=a_{\chi}(\pi)$. Define $D_{\chi}(\pi)$ to be the unique object in $\Irr(G_{n-al})$ such that $\pi\hookrightarrow \chi^{\times a}\times D_{\chi}(\pi)$. It is called the \textbf{(left highest) $\chi$-derivative} of $\pi$.
\end{definition}

\begin{lemma}\label{subquotient lem}
   Let $\chi\in\CC$. If $\pi$ is a subquotient of $\chi^{\times a} \times \rho$ such that $\rho\in \Irr$ is left $\chi$-reduced and $a_{\chi}(\pi)=a$. Then $\rho=D_{\chi}(\pi)$.
\end{lemma}
\begin{proof}
    By Lemma~\ref{Jac of left chi seq}, we obtain $\Jac_{\chi^{\times a}}(\pi)=\chi^{\times a} \boxtimes D_{\chi}(\pi)$ and 
    \[
    \Jac_{\chi^{\times a}}(\chi^{\times a} \times \rho)=\chi^{\times a} \boxtimes \rho,
    \]
    which implies $\rho=D_{\chi}(\pi)$ since Jacquet functor is exact.
\end{proof}
\subsection{socle irreducibility of big derivatives}
In this subsection, we recall a result of Chan~\cite{Ch24} concerning the socle irreducibility of big derivatives. We first recall the definition of big derivative. Let $\sigma\in\Irr(G_t)$ and $\pi\in \Fin(G_n)$. Assume $t\leq n$.
\begin{definition}[{\cite[Definition 8.1]{Ch24}}]
    \begin{enumerate}
        \item Define the (left) big derivative to be the $G_{n-t}$ representation
        \[
      \BW_{\sigma}(\pi):=  \Hom_{G_t}(\sigma, r_{t,n-t}(\pi)).
        \]
        \item Define the (right) big derivative to be the $G_{n-t}$ representation
        \[
        \BD_{\sigma}(\pi):=\Hom_{G_t}(\sigma, \ov{r}_{t,n-t}(\pi)).
        \]
    \end{enumerate}
\end{definition}
The left and right derivative are related by the cohomological duality functor as follows. Let $\CS(G_n)$ be the space of Bruhat-Schwartz functions on $G_n$. We regard it as a representation of $G_n$ by right translation. Let $\Omega\in M(\CC)$ such that $|\Omega|=n$. Assume that the homological dimension of $\Fin(\Omega)$ is $d$. For $\pi\in \Fin(\Omega)$, define
\[
\mathscr{D}(\pi):=\Ext^d_{G_n}(\pi, \CS(G_n)).
\]
The left translation on $\CS(G_n)$ induces an action of $G_n$ on $\mathscr{D}(\pi)$. It is proved in~\cite[Page 102]{Ber92} that the contravariant functor $\mathscr{D}$ is an involution, namely, $\mathscr{D}^2=\id$. In particular, $\mathscr{D}$ is exact and sends irreducible representations to irreducible representations. 

We can extend $\mathscr{D}$ to be a functor in $\Fin(G_n)$ by Bernstein decomposition:
\[
 \mathscr{D}(\pi):= \bigoplus_{\substack{\Omega\in M(\CC)\\ |\Omega|=n}} \mathscr{D}(\mathrm{pr}_{\Omega}(\pi)),
\]
where $\mathrm{pr}_{\Omega}$ is the projection functor under~\eqref{Bernstein decomp eq}. Let $\mathscr{A}(\pi):= \mathscr{D}(\pi)^{\vee}$. Then $\mathscr{A}$ is a covariant functor on $\Fin(G_n)$, which is also an equivalence of categories. Moreover, it agrees with the Aubert-Zelevinsky dual \cite{Ze80, Au95} in the Grothendieck group level, see the work of Schneider-Stuhler \cite[Proposition IV.5.2]{SS97} and Bernstein-Bezrukavnikov-Kazhdan \cite[Section 3.2]{BBK18}.
\begin{lemma}\label{duality commutes induction lem}
    Let $\Lambda_1,\Lambda_2\in M(\CC)$. Let $\pi_1\in \Fin(\Lambda_1)$ and $\pi_2\in \Fin(\Lambda_2)$. Then $\pi_1\times\pi_2\in \Fin(\Lambda_1+\Lambda_2)$. There is a natural isomorphism 
    \[
    \mathscr{A}(\pi_1\times \pi_2)\simeq \mathscr{A}(\pi_2)\times \mathscr{A}(\pi_1).
    \]
\end{lemma}
\begin{proof}
    It follows from~\cite[Theorem 31(4)]{Ber92}, which claims
    \[
     \mathscr{D}(\pi_1\times \pi_2)\simeq \mathscr{D}(\pi_2)\times \mathscr{D}(\pi_1).
    \]
\end{proof}

\begin{lemma}\label{dualtiy commutes big derivative}
    Let $\sigma\in \Irr(G_t)$. Then for any $\pi\in\Fin(G_n)$, there is a natural isomorphism:
    \[
    \mathscr{A}(\BW_{\sigma}(\pi))\simeq \BD_{\mathscr{A}(\sigma)}(\mathscr{A}(\pi)).
    \]
\end{lemma}
\begin{proof}
    By Yoneda's Lemma, it suffices to show that for any $\rho\in \Fin(G_{n-t})$, there is an isomorphism
    \[
    \Hom_{G_{n-t}}\left(\rho,  \mathscr{A}(\BW_{\sigma}(\pi))\right)\simeq  \Hom_{G_{n-t}}\left(\rho,  \BD_{\mathscr{A}(\sigma)}(\mathscr{A}(\pi))\right)
    \]
    By Bernstein decomposition~\eqref{Bernstein decomp eq}, we can assume $\rho\in \Fin(\Lambda)$ for some $\Lambda\in M(\CC)$ such that $|\Lambda|=n-t$. Then
    \begin{align*}
         \Hom_{G_{n-t}}\left(\rho,  \mathscr{A}(\BW_{\sigma}(\pi))\right)&\simeq \Hom_{G_{n-t}}(\mathscr{A}^{-1}(\rho), \BW_{\sigma}(\pi))\\
        &\simeq \Hom_{G_n}( \mathscr{A}^{-1}(\rho)\times \sigma, \pi)\\
        & \simeq \Hom_{G_n}( \mathscr{A}(\sigma)\times \rho, \mathscr{A}(\pi))\simeq\Hom_{G_{n-t}}\left(\rho,  \BD_{\mathscr{A}(\sigma)}(\mathscr{A}(\pi))\right),
    \end{align*}
    where the second and last isomorphisms come from Frobenius reciprocity, and the third isomorphism is Lemma~\ref{duality commutes induction lem}.
\end{proof}

Recall that for any segment $\Delta=[x,y]_{\chi}$, there is an associated representation $\St(\Delta)$, which is the unique irreducible quotient of 
\[
 \nu_{\chi}^x\chi\times \cdots \times \nu_{\chi}^y \chi.
\]
\begin{proposition}[{\cite[Proposition 11.5]{Ch24}}]\label{big derivative socle irreducible}
       Let $\pi\in \Irr$.  Let $\Delta$ be a segment. Suppose that $\BD_{\St(\Delta)}(\pi)\neq 0$, then it is socle irreducible.
\end{proposition}
\begin{proof}
    Let $\Delta=[x,y]_{\chi}$, where $\chi$ is a cuspidal representation. Then, by definition, $\Delta$ is strongly $\nu_{\chi}^y\chi$-saturated. Thus, the result follows from~\cite[Proposition 11.5]{Ch24}.
\end{proof}
\subsection{A Multiplicity one detector}
In this subsection, we introduce a lemma that detects the multiplicity one property from the left highest derivative. 
\begin{lemma}\label{multiplicity one detector lem}
     Let $\pi\in \Irr(G_n)$, and let $\chi$ be a cuspidal representation of $G_l$. Assume $a=a_{\chi}(\pi)$. Let $\sigma\in \Fin(G_{n-al})$ such that
     \begin{enumerate}
         \item any irreducible component of $\sigma$ is left $\chi$-reduced;
         \item Jordan-H\"older multiplicity $ [\sigma:D_{\chi}(\pi)]\leq 1$.
     \end{enumerate}
     Then $[\chi^{\times a}\times \sigma:\pi]\leq 1$.
\end{lemma}
\begin{proof}
For any irreducible component $\rho \neq D_{\chi}(\pi)$ of $\sigma$, $\pi$ cannot occur as an irreducible component of $\chi^{\times a} \times \rho$ by Lemma~\ref{subquotient lem}. In other words,
\[
 [\chi^{\times a}\times \sigma:\pi]=[\sigma:D_{\chi}(\pi)]\cdot [\chi^{\times a}\times D_{\chi}(\pi):\pi].
\]
On the other hand, since $\Jac_{\chi^{\times a}}(\pi)=D_{\chi}(\pi)$, Lemma~\ref{Jac of left chi seq} implies that
\[
[\chi^{\times a}\times D_{\chi}(\pi):\pi]= 1.
\]
This completes the proof.
\end{proof}

\subsection{Jacquet modules of Speh representations}
In this subsection, we collect some facts about Speh representations and their Jacquet module.

Let $\chi$ be a cuspidal representation of $G_l$, and let $a,b$ be two positive integers. We can write the Speh representations by Langlands classification:
\begin{equation}\label{multi-seg speh eq}
    U(\chi;a,b)= LQ\left( \St(\Delta_1)\times  \St(\Delta_2)\times \cdots\times  \St(\Delta_b)   \right),
\end{equation}
where $\Delta_i=[\frac{b-a}{2}+1-i,\frac{a+b}{2}-i]_{\chi}:=[x_i,y_i]_{\chi}$. It is a special class of ladder representations introduced in~\cite{LM14}. Applying the formula for Jacquet modules of ladder representations in~\cite[Theorem 2.1]{KL12}, we obtain the following lemma.
\begin{lemma}\label{Speh Jacquet lem}
 Let $n=lab$, and let $t$ be a positive integer such that $0<t<n$. There is a natural isomorphism as representations of $G_{n-t}\times G_{t}$:
 \begin{align*}
      &r_{n-t,t}(U(\chi;a,b))
      =\bigoplus_{\substack{c_1>\cdots>c_b \\ \sum_{i=1}^b l(c_i-x_i+1)=t}} LQ(\St([c_1+1,y_1]_{\chi})\times \cdots\times \St([c_b+1,y_b]_{\chi}))\\
      &\boxtimes LQ(\St([x_1,c_1]_{\chi})\times \cdots\times \St([x_b,c_b]_{\chi})),
 \end{align*}
where we assume $\St([x,y]_{\chi})=0$ if $y<x-1$ and $\St([x,x-1]_{\chi})$ to be the trivial representation of $G_0$. In particular, the Jacquet module of Speh representations is multiplicity-free. 

\end{lemma}

Using the M\oe glin--Waldspurger's algorithm, we obtain an equivalent lemma in terms of the Zelevinsky classification:
\begin{equation}\label{Speh Zelevinsky}
    U(\chi;a,b)=Z(\Sigma_1+\cdots+\Sigma_a),
\end{equation}
where $\Sigma_i=[\frac{a-b}{2}+1-i,\frac{a+b}{2}-i]_{\chi}:=[x_i',y_i']_{\chi}$.
\begin{lemma}\label{Speh Jacquet lem zelevinsky}
 Let $n=lab$, and let $t$ be a positive integer such that $0<t<n$. There is a natural isomorphism as representations of $G_{n-t}\times G_{t}$:
 \begin{align*}
      &r_{n-t,t}(U(\chi;a,b))
      =\bigoplus_{\substack{c_1>\cdots>c_a \\ \sum_{i=1}^a l(y_i'-c_i)=t}}  Z([x_1',c_1]_{\chi}+ \cdots+ [x_a',c_a]_{\chi})\\
      &\boxtimes Z([c_1+1,y_1']_{\chi}+ \cdots + [c_a+1,y_a']_{\chi}),
 \end{align*}
where we assume $Z([x,y]_{\chi})=0$ if $y<x-1$ and $Z([x,x-1]_{\chi})$ to be the trivial representation of $G_0$.
\end{lemma}

\subsection{Godement-Jacquet zeta integral}
In this section, we recall the integral representation of the standard $L$-function \`a la Godement-Jacquet. 

Let $(n,m)$ be a pair of integers such that $m\geq n$. Let $\pi\in \Irr(G_n)$. Given $\phi\in S_{n,m}$, $v\in \pi$, $\lambda\in \pi^{\vee}$, and $s\in \BC$, we can associate a zeta integral
\[
Z(\phi,\lambda,v;s):=\int_{G_n} \phi( g\mid 0) \lambda(g\cdot v) |\Nrd(g)|_F^{s+\frac{dn-1}{2}} \, dg.
\]
It has the following basic properties.
\begin{proposition}\label{GJ proposition}
    \begin{enumerate}
        \item The zeta integral $Z(\phi,\lambda,v;s)$ is absolutely convergent when $\mathrm{Re} s$ is sufficiently large, and it admits a meromorphic continuation to the whole complex plane, which is also denoted by $Z(\phi,\lambda,v;s)$.  
        \item 
        \[
         Z^{\sharp}(\phi,\lambda,v;s):=\frac{Z(\phi,\lambda,v;s)}{L(s,\pi)}
        \]
        is a holomorphic function.
        \item It induces a non-zero $G_n\times \ov{P_n'}$-equivariant map $Z^{\sharp}(\cdot,\cdot,\cdot;s): \omega_{n,m} \lra \pi\boxtimes \pi^{\vee}\otimes \xi_s$, where 
       the action of $ \ov{P_n'}$ factors through $\ov{P_n'}\lra G_n'\times G_{m-n}'$ and
       \begin{equation}
           \xi_s=\begin{cases}
               |\Nrd|^{s+\frac{d(n-m)-1}{2}}  & \text{ on } G_n\\
               |\Nrd|^{-s+\frac{1}{2}}    & \text{ on } G_n'\\
               \nu^{\frac{n}{2}}       & \text{ on } G_{m-n}'.
           \end{cases}
       \end{equation}
       Here, on the right hand side, $G_n$ acts on $\pi$ and $G_n'$ acts on $\pi^{\vee}$.
    \end{enumerate}
\end{proposition}
\begin{proof}
    Since 
    \[
    S_{n,m}\lra S_{n,n} \quad \phi\longmapsto \left(X\mapsto \phi(X\mid 0) \right)
    \]
    is surjective, the proposition follows from~\cite{GJ72}.
\end{proof}
By (3) of the Proposition and Frobenius reciprocity, we obtain a non-zero element in
\[
 Z^{\sharp}\in \Hom_{G_n\times G_m'}(\omega_{n,m}\otimes \pi^{\vee}, \mathbf{1}_{m-n}\times \pi^{\vee})
\]
when specializing to $s=s_0:=\frac{1+d(m-n)}{2}$. 

Let $\widetilde{R}_n$ be the subrepresentation of $\omega_{n,m}$ consisting of functions supported in the full rank locus of $M_{n,m}(\mathrm{D})$.
\begin{lemma}\label{GJ integral on boundary lem}
    If $L(s,\pi)$ is not holomorphic at $s=s_0$, then $Z^{\sharp}$ restricted to $\widetilde{R}_n\otimes \pi^{\vee}$ is zero.
\end{lemma}
\begin{proof}
    We observe that 
    \[
    \int_{G_n} \phi( g\mid 0) \lambda(g\cdot v) |\Nrd(g)|_F^{s+\frac{dn-1}{2}} \, dg
    \]
    is absolutely convergent for any $v\in\pi,\lambda\in\pi^{\vee}$ and $s\in\BC$ when $\phi\in \widetilde{R}_n$. Thus, it equals $Z(\phi,\lambda,v;s)$ for any $s\in \BC$. When $L(s,\pi)$ is not holomorphic at $s=s_0$, we have $Z^{\sharp}(\phi,\lambda,v;s_0)=0$ for any $v\in\pi$ and $\lambda\in\pi^{\vee}$. This completes the proof.
\end{proof}

\section{Rank filtration and Kudla's filtration}
\subsection{Rank filtration and off the boundary condition}
In this subsection, we recall the rank filtration and the off-boundary condition in~\cite{Ming08}.

There is a natural decreasing filtration on $(\sigma_{n,m},S_{n,m})$ induced by ranks of matrices. Let $R_k$ be the subrepresentation of $S_{n,m}$ consisting of functions supported on matrices of rank greater or equal to $k$. Then, Since the $G_n\times G_m'$ action on $M_{n,m}(\mathrm{D})$ preserves the rank, we obtain a $G_n\times G_m'$-invariant decreasing filtration:
\[
0=R_{e+1}\subset R_e \subset \cdots\subset R_0=\sigma_{n,m},
\]
where $e=\min{\{n,m\}}$. Moreover, the graded piece is described by the following lemma.
\begin{lemma}[{\cite[section 2]{Ming08}}]\label{rank filtration}
    As representations of $G_n\times G_m'$, 
    \[
    R_k/R_{k+1}\simeq \ind_{P_k\times \ov{P'_k}}^{G_n\times G_m'}(\rho_k \otimes \eta_k),
    \]
    where $\eta_k$ is a character of the standard Levi subgroup of $P_k\times \ov{P'_k}$, given by
    \begin{equation}
      \eta_k=  \begin{cases}
             \nu^{\frac{k-n}{2}} & \text{ on } G_k\\
             \nu^{\frac{k}{2}} & \text{ on } G_{n-k}\\
             \nu'^{\frac{m-k}{2}}& \text{ on } G_k' \\
             \nu'^{\frac{-k}{2}} & \text{ on } G_{m-k}'
        \end{cases}.
    \end{equation}
\end{lemma}
\begin{definition}[Definition 2.1,~\cite{Ming08}]
    We say $\pi\in \Irr(G_n)$ is off the boundary if 
    \[
    \Hom_{G_n}(R_k/R_{k+1},\pi)=0
    \]
    for any $0\leq k<e$.
\end{definition}

If $\pi$ is off the boundary, then $\Hom_{G_n}(\sigma_{n,m}/R_e,\pi)=0$. Applying the left exact functor $\Hom_{G_n}(\cdot,\pi)_{\sm}$ to the short exact sequence
\[
 0\lra R_e \lra \sigma_{n,m}\lra \sigma_{n,m}/R_e \lra 0,
\]
we obtain the exact sequence
\[
0\lra \Hom_{G_n}(\sigma_{n,m}/R_e,\pi)_{\sm}=0\lra \Hom_{G_n}(\sigma_{n,m},\pi)_{\sm}\lra \Hom_{G_n}(R_e,\pi)_{\sm}.
\]

\begin{lemma}\label{lift off boundary lemma}
    Assume $m\geq n$, and assume $\pi\in \Irr(G_n)$ is off the boundary. Then 
    \[
    \ov{\Theta}_{n,m}(\pi)\hookrightarrow  \nu_{m-n}'^{\frac{n}{2}} \times (\pi\cdot \nu_n'^{\frac{n-m}{2}}).
    \]
\end{lemma}
\begin{proof}
    Since $\pi$ is off the boundary, 
    \[
    \ov{\Theta}_{n,m}(\pi)\hookrightarrow \Hom_{G_n}(R_n, \pi)_{\sm}\simeq \Hom_{G_n}\left( \ind_{G_n\times \ov{P'_n}}^{G_n\times G_m'}(\rho_n \otimes \eta_n),\pi\right)_{\sm}.
    \]
    As representations of $G_n'$, there is an isomorphism $(\rho_n\otimes \pi^{\vee})_{G_n}\simeq \pi^{\vee}$. Therefore,
    \[
    \ov{\Theta}_{n,m}(\pi)\hookrightarrow \ind_{\ov{P'_n}}^{G_n'}( \pi^{\vee}\cdot \nu'^{\frac{m-n}{2}}\boxtimes \nu'^{\frac{-n}{2}})^{\vee}\simeq \ind_{\ov{P'_n}}^{G_n'}( \pi\cdot \nu'^{\frac{n-m}{2}}\boxtimes \nu'^{\frac{n}{2}}).
    \]
    This completes the proof.
\end{proof}

\begin{lemma}\label{down off boundary lem}
    Assume $m<n$, and assume $\pi\in \Irr(G_n)$ is off the boundary. Then
    \[
   \ov{\Theta}_{n,m}(\pi)\hookrightarrow \BW_{\nu_{n-m}^{\frac{m}{2}}}(\pi)\cdot \nu'^{\frac{n-m}{2}}.
    \]
\end{lemma}
\begin{proof}
    Since $\pi$ is off the boundary,
    \[
    \ov{\Theta}_{n,m}(\pi)\hookrightarrow \Hom_{G_n}(R_n, \pi)_{\sm}\simeq \Hom_{G_n}\left( \ind_{P_m\times \ov{G'_m}}^{G_n\times G_m'}(\rho_m \otimes \eta_m),\pi\right)_{\sm}.
    \]
    As representations of $G_m'$, there is an isomorphism $\Hom_{G_m}(\rho_m,\pi)_{\sm}\simeq \pi$. Therefore,
    \[
     \ov{\Theta}_{n,m}(\pi)\hookrightarrow \Hom_{L_{m,n-m}}(\rho_m\otimes\eta_m, \ov{r}_{m,n-m}(\pi))_{\sm}\simeq \Hom_{G_{n-m}}(\nu^{\frac{m}{2}}, r_{n-m,m}(\pi))\nu'^{\frac{n-m}{2}}.
    \]
    This completes the proof.
\end{proof}

\begin{proposition}\label{off boundary multiplicity one prop}
    Suppose that $\pi \in \Irr(G_n)$ is off the boundary. Then, for any positive integer $m$, the representation $\ov{\Theta}_{n,m}(\pi)$ is socle-irreducible whenever it is non-zero.
\end{proposition}
\begin{proof}
    When $m\geq n$, it follows from Lemma~\ref{lift off boundary lemma} and the fact that $\nu_{m-n}'^{\frac{n}{2}} \times (\pi\cdot \nu_n'^{\frac{n-m}{2}})$ is socle irreducible. See~\cite[Corollary 4.9]{LM16}.

    Now, suppose $m<n$. Since $\mathscr{A}$ is a covariant equivalence of category, it suffices to show that $\mathscr{A}(\ov{\Theta}_{n,m}(\pi))$ is socle irreducible if it is non-zero. Lemma~\ref{down off boundary lem} and Lemma~\ref{dualtiy commutes big derivative} shows that
    \[
   \mathscr{A}(\ov{\Theta}_{n,m}(\pi)) \hookrightarrow \BD_{\mathscr{A}(\nu_{n-m}^{\frac{m}{2}})}(\mathscr{A}(\pi))\cdot \nu^{\frac{n-m}{2}}.
    \]
    By M\oe glin-Waldspurger's algorithm, we have
    \[
   \mathscr{A}(\nu_{n-m}^{\frac{m}{2}})= \nu_{n-m}^{\frac{m}{2}} \St_{n-m},
    \]
    where $\St_{n-m}=\St(n-m;\mathbf{1}_1)$ is the Steinberg representation of $G_{n-m}$. Consequently, the result follows from Proposition~\ref{big derivative socle irreducible}.
\end{proof}

\subsection{Kudla's filtration}
In Lemma~\ref{reducedness characterization lem}, we observe that an important way to prove the cuspidal reducedness is by Jacquet functor. Thus, to prove certain cuspidal reducedness of big theta, it is natural to investigate the Jacquet functor of the representation $(\sigma_{n,m},S_{n,m})$, which has a decreasing filtration due to Kudla:
\[
0=J_{k+1}\subset J_k\subset \cdots\subset J_0= r^{G_n}_{t,n-t}(\sigma_{n,m}),
\]
where $k=\min{\{t,m\}}$ and $0\leq t\leq n$. Let $\tau_i\simeq J_i/J_{i+1}$. The following lemma describes these graded pieces.
\begin{proposition}[{\cite[Proposition 3.2]{Ming08}}]\label{Kudla filtration left}
    As representations of $L_{t,n-t}\times G_m'$, there is a natural isomorphism 
   \[
\tau_i\simeq
\ind_{P_{t-i,i}\times G_{n-t}\times P'_{i,m-i}}^{L_{t,n-t}\times G'_m}
\bigl(\xi_{t,i}\otimes\rho_i\otimes\sigma_{n-t,m-i}\bigr),
\]
where $\rho_i$ is defined by~\eqref{rho def eq}, and $\xi_{t,i}$ is the character
\[
\xi_{t,i}=
\begin{cases}
\nu^{\frac{2m-n+t-i}{2}}&\text{on }G_{t-i},\\
\nu^{\frac{2m-n+2t-i}{2}}&\text{on }G_i,\\
\nu^{\frac t2}&\text{on }G_{n-t},\\
\nu'^{\frac{-m-2t+i}{2}}&\text{on }G'_i,\\
\nu'^{\frac{-2t+i}{2}}&\text{on }G'_{m-i}.
\end{cases}
\]
\end{proposition}
\begin{remark}\label{rank filtration dual rem}
    When $t=n$, Kudla's filtration in Proposition~\ref{Kudla filtration left} is identical to the rank filtration on the Fourier dual side. More precisely, there is a non-degenerate bilinear pairing
\[
M_{n,m}(\mathrm{D})\times M_{m,n}(\mathrm{D})\lra F \quad \langle X,Y\rangle:= \Trd_{D/F}(XY).
\]
We fix mutually dual Haar measures, and define the Fourier transform
\begin{equation}\label{Fourier transform eq}
    \begin{split}
            \CF:&\CS(M_{n,m}(D))\lra \CS(M_{m,n}(D))\\
    f\mapsto &\CF(f)(Y)=\int_{M_{n,m}(D)} f(X)\psi(\langle X,Y\rangle)dX.
    \end{split}
\end{equation}
Then Kudla's filtration coincides with the rank filtration on $\CS(M_{m,n}(D))$ under the action $\sigma_{m,n}(g',g)\nu^{m}(g) \nu^{-n}(g')$.
\end{remark}
Similarly, we can consider the Jacquet functor of $\sigma_{n,m}$ with respect to $G_m'$:
\[
0 = J_{k'+1}'\subset J'_{k'} \subset \cdots \subset J_0'= \ov{r}^{G_m'}_{t,m-t}(\sigma_{n,m}),
\]
where $k'=\min{\{t,n\}}$ and $0\leq t\leq m$. Let $\tau_i'\simeq J_i'/J_{i+1}'$. The following lemma describes these graded pieces.
\begin{proposition}[{\cite[Proposition 3.3]{Ming08}}]\label{Kudal filtration right}
    As representations of $G_n\times L'_{t,m-t}$, there is a natural isomorphism
   \[
\tau'_i\simeq
\ind_{P_{n-i,i}\times P'_{i,t-i}\times G'_{m-t}}^{G_n\times L'_{t,m-t}}
\bigl(\sigma_{n-i,m-t}\otimes\rho_i\otimes\xi'_{t,i}\bigr),
\]
where $\rho_i$ is defined by~\eqref{rho def eq}, and $\xi'_{t,i}$ is the character
\[
\xi'_{t,i}=
\begin{cases}
\nu^{\frac{2t-i}{2}}&\text{on }G_{n-i},\\
\nu^{\frac{n+2t-i}{2}}&\text{on }G_i,\\
\nu^{\frac{-2n+m-2t+i}{2}}&\text{on }G'_i,\\
\nu^{\frac{m-2n-t+i}{2}}&\text{on }G'_{t-i},\\
\nu^{-t/2}&\text{on }G'_{m-t}.
\end{cases}
\]
\end{proposition}

Now, we use Kudla's filtration to compute the big theta in terms of the left highest derivative. 
\begin{proposition}\label{big theta relative to highest derivative}
  Let $\pi\in\Irr(G_n)$. And let $\chi$ be a cuspidal representation of $G_l$ such that $\chi\neq \nu_1^{\frac{2m-n+1}{2}}$. Suppose that $\pi\hookrightarrow \chi^{\times a}\times \rho$, where $\rho\in \Irr(G_{n-la})$. 
  \begin{enumerate}
      \item If $m<la$, then $\ov{\Theta}_{n,m}(\pi)=0$;
      \item If $m\geq la$, then 
      \[\ov{\Theta}_{n,m}(\pi)\hookrightarrow (\nu^{\frac{n-m}{2}}\chi)^{\times a} \times \ov{\Theta}_{n-la,m-la}(\nu^{-\frac{la}{2}}\rho)\nu'^{\frac{la}{2}}.
      \]
     Recall when $nm=0$, for $\sigma\in \Irr(G_n)$, we set 
     \[
     \ov{\Theta}_{n,m}(\sigma)\simeq\begin{cases}
          0 &\text{ if } \sigma \not\simeq \mathbf{1}_n;\\
           \mathbf{1}_m &\text{ if } \sigma \simeq \mathbf{1}_n.
     \end{cases}
     \]
  \end{enumerate}
\end{proposition}
\begin{proof}
  Since $\pi\hookrightarrow \chi^{\times a}\times \rho$, we have
 \[
  \ov{\Theta}_{n,m}(\pi)\hookrightarrow \Hom_{G_n}(\sigma_{n,m}, \chi^{\times a}\times \rho)_{\sm}\simeq \Hom_{L_{al,n-al}}\left(r^{G_n}_{al,n-al}(\sigma_{n,m}), \chi^{\times a} \boxtimes \rho\right)_{\sm}.
 \]
    According to the proof of~\cite[Lemma 4.1]{Ming08}, we know that
    \[
    \Hom_{L_{la,n-la}}(\tau_i, \chi^{\times a} \boxtimes \rho)=0 
    \]
    for any $i<al$. We remark that the condition that $\chi\neq  \nu_1^{\frac{2m-n+1}{2}}$ is used to ensure that when $i=al-1$, it is zero. This proves (1).
    
     Assume $m\geq al$. Then,
    \[
     \ov{\Theta}_{n,m}(\pi)\hookrightarrow \Hom_{L_{al,n-al}}(\tau_{al},\chi^{\times a}\boxtimes \rho)_{\sm}.
    \]
    By Proposition~\ref{Kudla filtration left} and the fact that as representations of $G_{al}'$, $\Hom_{G_{al}}(\rho_{al},\chi^{\times a})_{\sm}\simeq \chi^{\times a}$, one can deduce
    \begin{align*}
 &  \Hom_{L_{al,n-al}}(\tau_{al},\chi^{\times a}\boxtimes \rho)_{\sm}    \\
   &\simeq  \Hom_{L_{la,n-la}}\left( \ind_{ P'_{al}}^{ G'_m}
\bigl(\xi_{al,al}\otimes\rho_{al}\otimes\sigma_{n-al,m-al}\bigr) , \chi^{\times a}\boxtimes \rho  \right)_{\sm}\\
&\simeq \ind^{G_m'}_{P_{al}'}\left( \Hom_{G_{al}}(\rho_{al},\chi^{\times a}\nu^{\frac{-2m+n-al}{2}})_{\sm}\nu'^{\frac{m+al}{2}}\boxtimes \Hom_{G_{n-al}}(\xi_{al,al}\otimes\sigma_{n-al,m-al},\rho)_{\sm} \right)\\
&\simeq (\nu^{\frac{n-m}{2}}\chi)^{\times a} \times \ov{\Theta}_{n-la,m-la}(\nu^{-\frac{la}{2}}\rho)\nu'^{\frac{la}{2}}.
    \end{align*} 
    If $n=al$ , then $\rho$ and $\sigma_{n-al,m-al}$  are trivial representations. Thus, $\ov{\Theta}_{n-la,m-la}(\nu^{-\frac{la}{2}}\rho)$ is the trivial representation of $G_{m-al}'$.
\end{proof}

\section{Big theta preserves cuspidal reducedness}
To prove cuspidal reducedness, we compute the Jacquet module of $\ov{\Theta}_{n,m}(\pi)$. This computation relies on the following lemma.
\begin{lemma}\label{Jacquet big theta lemma}
    Let $\pi$ be a representation of $G_n$. Then, as representations of $L_{l,m-l}'$, 
    \[
    r^{G_m'}_{l,m-l}(\ov{\Theta}_{n,m}(\pi))\simeq \Hom_{G_n}\left(\ov{r}^{G_m'}_{l,m-l}(\sigma_{n,m}),\pi\right)_{\sm}.
    \]
\end{lemma}
\begin{proof}
    Since $\ov{\Theta}_{n,m}(\pi)\simeq (\sigma_{n,m}\otimes\pi^{\vee})_{G_n}^{\vee}$, we have
    \[
    r^{G_m'}_{l,m-l}(\ov{\Theta}_{n,m}(\pi))\simeq \left(\ov{r}^{G_m'}_{l,m-l}(\sigma_{n,m}\otimes\pi^{\vee})_{G_n}\right)^{\vee}.
    \]
    Since the $G_n$-action and $G_m'$-action commute on $\sigma_{n,m}\otimes \pi^{\vee}$, it is isomorphic to
    \[
    \left((\ov{r}^{G_m'}_{l,m-l}(\sigma_{n,m})\otimes\pi^{\vee})_{G_n}\right)^{\vee}\simeq \Hom_{G_n}\left(\ov{r}^{G_m'}_{l,m-l}(\sigma_{n,m}),\pi\right)_{\sm}.
    \]
\end{proof}

\begin{theorem}\label{preserve reducedness thm}
    Let $\pi\in\Irr(G_n)$. And let $\chi$ be a cuspidal representation of $G_l$ such that $\chi\neq \nu_1^{\frac{n+1}{2}}$. If $\pi$ is left $\chi$-reduced, then any irreducible component in $\ov{\Theta}_{n,m}(\pi)$ is left $\nu^{\frac{n-m}{2}}\chi$-reduced.
\end{theorem}
\begin{proof}
    By the exactness of the Jacquet functor and Lemma~\ref{reducedness characterization lem}, it suffices to prove
    \begin{equation}\label{Jacquet trivial eq}
         \Jac_{\nu^{\frac{n-m}{2}}\chi}(\ov{\Theta}_{n,m}(\pi))=0.
    \end{equation}
    By Lemma~\ref{Jacquet big theta lemma}, we have
    \[
    r^{G_m'}_{l,m-l}(\ov{\Theta}_{n,m}(\pi))\simeq \Hom_{G_n}\left(\ov{r}^{G_m'}_{l,m-l}(\sigma_{n,m}),\pi\right)_{\sm}.
    \]
    By Proposition~\ref{Kudal filtration right}, we consider
    \begin{align*}
            \Hom_{G_n}\left(\tau_i',\pi\right)_{\sm}&\simeq \Hom_{G_n}\left(  \ind_{P_{n-i,i}\times P'_{i,l-i}\times G'_{m-l}}^{G_n\times L'_{l,m-l}}
\bigl(\sigma_{n-i,m-l}\otimes\rho_i\otimes\xi'_{l,i}\bigr),\pi\right)_{\sm}\\
&\simeq \Hom_{L_{n-i,i}}\left(  \ind_{ P'_{i,l-i}\times G'_{m-l}}^{ L'_{l,m-l}}
\bigl(\sigma_{n-i,m-l}\otimes\rho_i\otimes\xi'_{l,i}\bigr),\ov{r}_{n-i,i}(\pi)\right)_{\sm}
    \end{align*}
Let $\alpha\boxtimes \beta$ be an irreducible component of $\ov{r}_{n-i,i}(\pi)$, where $\alpha\in\Irr(G_{n-i})$ and $\beta\in \Irr(G_i)$. Then
\begin{align*}
   & \Hom_{L_{n-i,i}}\left(  \ind_{ P'_{i,l-i}\times G'_{m-l}}^{ L'_{l,m-l}}
\bigl(\sigma_{n-i,m-l}\otimes\rho_i\otimes\xi'_{l,i}\bigr),\alpha\boxtimes \beta)\right)_{\sm}\\
\simeq & \Hom_{G_i}\left(\ind^{G_l'}_{P_{i,l-i}'}(\nu^{\frac{-2n+m-2l+i}{2}}\rho_i\boxtimes \nu^{\frac{m-2n-l+i}{2}}),\nu^{\frac{-n+i-2l}{2}}\beta\right)_{\sm}\boxtimes \ov{\Theta}_{n-i,m-l}(\alpha\cdot\nu^{\frac{i-2l}{2}})\nu'^{\frac{l}{2}}.
\end{align*}
Since as representations of $G_i'$, $(\rho_i\otimes \beta^{\vee})_{G_i}\simeq \beta^{\vee}$, we obtain
\begin{align*}
 &\Hom_{G_i}\left(\ind^{G_l'}_{P_{i,l-i}'}(\nu^{\frac{-2n+m-2l+i}{2}}\rho_i\boxtimes \nu^{\frac{m-2n-l+i}{2}}),\nu^{\frac{-n+i-2l}{2}}\beta\right)_{\sm}\\
  \simeq&\left(\nu^{\frac{m-n}{2}}\beta^{\vee}\times \nu^{\frac{m-2n-l+i}{2}}\right)^{\vee}  \simeq \nu^{\frac{n-m}{2}}\beta \times \nu^{\frac{-m+2n+l-i}{2}}.
\end{align*}
When $0 < i < l$, it is parabolically induced and thus does not lie in the cuspidal block $\Fin([\nu^{\frac{n-m}{2}}\chi])$. When $i=l$, it is isomorphic to 
\[
\nu^{\frac{n-m}{2}}\beta.
\]
Since $\pi$ is left $\chi$-reduced, $\beta\neq \chi$, which implies $\nu^{\frac{n-m}{2}}\beta\notin \Fin([\nu^{\frac{n-m}{2}}\chi])$. When $i=0$, it is isomorphic to $\nu^{\frac{-m+2n+l}{2}}$. It does not lie in the block $\Fin([\nu^{\frac{n-m}{2}}\chi])$ unless $l=1$. However, at this time, since $\chi\neq \nu_1^{\frac{n+1}{2}}$, we have
\[
 \nu^{\frac{n-m}{2}}\chi\neq \nu^{\frac{-m+2n+l}{2}}=\nu^{\frac{-m+2n+1}{2}}.
\]
In conclusion, 
\[
 \nu^{\frac{n-m}{2}}\beta \times \nu^{\frac{-m+2n+l-i}{2}}\notin \Fin([\nu^{\frac{n-m}{2}}\chi])
\]
for $i\in\{0,1\cdots,l\}$. Consequently, we obtain~\eqref{Jacquet trivial eq}.
\end{proof}

\section{Proof of the multiplicity one Theorem}
Recall the set of critical cuspidal representations defined in~\cite{Ming08}
\[
\mathscr{J}_{n,m}:=\left\{\nu_1^{\frac{n+1}{2}},\nu_1^{\frac{2m-n+1}{2}}\right\}.
\]
The proof starts from a \textbf{binary lemma} due to M\'inguez which ensures that for $\pi\in \Irr(G_n)$, either $\pi$ is off the boundary, or there exists a non-critical cuspidal representation $\chi$ such that $\pi$ is not left $\chi$-reduced.
\begin{lemma}\label{binary lem}
    Let $\pi\in\Irr(G_n)$. For any positive integer $m$, either $\pi$ is off the boundary, or there exists a cuspidal $\chi\notin \mathscr{J}_{n,m}$ such that $\pi$ is not left $\chi$-reduced.
\end{lemma}
\begin{proof}
    It follows from the proof of~\cite[Theorem 5.1]{Ming08}. We remark that although the theorem was stated for $n\leq m$ in~\cite{Ming08}, the proof works for any $m$.
\end{proof}

We are ready to prove the multiplicity one result.
\begin{theorem}\label{multiplicity one thm}
  Let $(n,m)$ be a pair of positive integers.  Let $\pi$ be an irreducible representation of $G_n$. If $\theta_{n,m}(\pi)\neq 0$, then 
    \[
    [\Theta_{n,m}(\pi):\theta_{n,m}(\pi)]=1.
    \]
\end{theorem}
\begin{proof}
By Howe duality, $\theta(\pi)^{\vee}$ is the unique irreducible subrepresentation of $\Theta(\pi)^{\vee}$. Since twisting by a character will not affect the Jordan-H\"older multiplicity, it suffices to prove that $\ov{\Theta}(\pi)$ is socle irreducible.

    We argue by induction on the integer $n\cdot m$. When $nm=0$, the statement is obviously true. Suppose that the theorem holds for any pair $(i,j)$ such that $ij<nm$, and proceed to prove the theorem for $(n,m)$.  Assume $\ov{\Theta}_{n,m}(\pi)\neq 0$.

    \textbf{Case 1:} When $\pi$ is off the boundary, then the theorem follows from Proposition~\ref{off boundary multiplicity one prop}.

    \textbf{Case 2:} When $\pi$ is on the boundary, then by binary Lemma~\ref{binary lem}, there exists a cuspidal representation $\chi\notin \mathscr{J}_{n,m}$ of $G_l$, such that
    \[
    \pi\hookrightarrow \chi^{\times a}\times \rho,
    \]
    where $a=a_{\chi}(\pi)$ and $\rho=D_{\chi}(\pi)$. By Proposition~\ref{big theta relative to highest derivative}, we know $m\geq la$ and 
     \[\ov{\Theta}_{n,m}(\pi)\hookrightarrow (\nu^{\frac{n-m}{2}}\chi)^{\times a} \times \ov{\Theta}_{n-la,m-la}(\nu^{-\frac{la}{2}}\rho)\nu'^{\frac{la}{2}}.
      \]
      When $m=la$, then by Lemma~\ref{cuspidal product irr lem}, $\ov{\Theta}_{n,m}(\pi)$ is irreducible. Thus, we can assume $m>la$.  \\
      \textbf{Case 2 (a). Assume $n>la$.} Since $\nu^{-\frac{la}{2}}\rho$ is left $\nu^{-\frac{la}{2}}\chi$-reduced and 
      \[
      \nu^{-\frac{la}{2}}\chi\notin \mathscr{J}_{n-la,m-la}= \nu^{-\frac{la}{2}}\mathscr{J}_{n,m},
      \]
       by Theorem~\ref{preserve reducedness thm}, any irreducible component in $\ov{\Theta}_{n-la,m-la}(\nu^{-\frac{la}{2}}\rho)\nu'^{\frac{la}{2}}$ is left $\nu^{\frac{n-m}{2}}\chi$-reduced. \\
       \textbf{Case 2 (b). Assume $n=la$.} Then $\ov{\Theta}_{n-la,m-la}(\nu^{-\frac{la}{2}}\rho)\nu'^{\frac{la}{2}}= \nu_{m-n}'^{\frac{n}{2}}$. If it is not left $\nu^{\frac{n-m}{2}}\chi$-reduced, then $\chi$ is a character of $G_1$ and 
       \[
       \nu_1^{\frac{n-m}{2}}\chi=\nu_1^{\frac{2n-m+1}{2}} \Longleftrightarrow \chi\simeq \nu_1^{\frac{n+1}{2}},
       \]
       which contradicts to the assumption $\chi\notin \mathscr{J}_{n,m}$. In conclusion, any irreducible component in $\ov{\Theta}_{n-la,m-la}(\nu^{-\frac{la}{2}}\rho)\nu'^{\frac{la}{2}}$ is left $\nu^{\frac{n-m}{2}}\chi$-reduced.
       
       Let the unique irreducible subrepresentation of $\ov{\Theta}_{n,m}(\pi)$ be $\ov{\theta}_{n,m}(\pi)$. Then~\cite[Proposition 4.4]{Ming08} computes that 
       \[
       D_{ \nu^{\frac{n-m}{2}}\chi}(\ov{\theta}_{n,m}(\pi))=  \ov{\theta}_{n-la,m-la}(\nu^{-\frac{la}{2}}\rho)\nu'^{\frac{la}{2}} \text{ and } a_{ \nu^{\frac{n-m}{2}}\chi}(\ov{\theta}_{n,m}(\pi))=a.
       \]
       Then Lemma~\ref{multiplicity one detector lem} with the inductive hypothesis 
      \[
      [\ov{\Theta}_{n-la,m-la}(\nu^{-\frac{la}{2}}\rho):\ov{\theta}_{n-la,m-la}(\nu^{-\frac{la}{2}}\rho)]=1
      \]
      implies that
    \[
    [\ov{\Theta}_{n,m}(\pi):\ov{\theta}_{n,m}(\pi)]\leq [(\nu^{\frac{n-m}{2}}\chi)^{\times a} \times \ov{\Theta}_{n-la,m-la}(\nu^{-\frac{la}{2}}\rho)\nu'^{\frac{la}{2}}:\ov{\theta}_{n,m}(\pi)]=1.
    \]
\end{proof}
\section{Big theta for Arthur-type representations}
In this section, we compute the big theta for Arthur-type representations. Motivated by theorem~\ref{L-function holomorphic big theta thm}, the first natural question is to identify the exceptional Arthur-type representations.
\subsection{Arthur-type representations of inner forms}\label{Arthur-type sec}
In the introduction, we recall the definition of the Arthur-type representation of the split group $\GL_n(F)$. We shall discuss in more detail the Arthur-type representations of the inner form $G_n=\GL_n(\mathrm{D})$, see~\cite{Cai23} for a survey. 

We first introduce a class of representations whose corresponding A-parameter
\[
 \phi: W_F\times \SL_2(\BC)\times \SL_2(\BC)\lra \widehat{G_n}=\GL_{nd}(\BC).
\]
is irreducible. For $\mathrm{D}\neq F$, it is generally not a single Speh block, but rather a product of several Speh representations. 

Let $\chi$ be a unitary cuspidal representation of $G_l$, and let $a,b$ be positive integers. Denote $n=lab$. Define $u(\chi;a,b)$ as the Langlands quotient
\[
LQ\left(|\Nrd|_F^{\frac{b-1}{2}}\St(a,\chi)\times \cdots\times |\Nrd|_F^{\frac{1-b}{2}}\St(a,\chi)\right).
\]
Then it is a representation of $G_{n}$. 

Under the Jacquet-Langlands correspondence, a cuspidal representation $\chi\in\Irr(G_l)$ will correspond to an essential discrete series $\pi$ of $\GL_{ld}(F)$. More precisely, there exists a cuspidal representation $\rho$ of $\GL_{ld/s_{\chi}}(F)$ such that $\pi$ is the unique irreducible quotient of 
\[
  |\Nrd|^{\frac{1-s_{\chi}}{2}}_F\cdot \rho \times \cdots \times |\Nrd|^{\frac{s_{\chi}-1}{2}}_F\cdot \rho .
\]
We denote $\rho$ by $\mathrm{JLC}(\chi)$. Then the A-parameter of $u(\chi;a,b)$ is
\[
\phi= \tau_{\rho}\boxtimes \Sym^{s_{\chi}a-1}(\BC^2)\boxtimes \Sym^{b-1}(\BC^2),
\]
which is irreducible. Conversely, any Arthur-type representation with an irreducible A-parameter is isomorphic to $u(\chi;a,b)$ for some unitary cuspidal representation $\chi$ and positive integers $a,b$. We remark that when $\mathrm{D}\neq F$, many Arthur packets are empty.

An important observation is that we can write $u(\chi;a,b)$ as a product of Speh representations. Let 
\[
b_j=1+\lfloor\frac{b-1-j}{s_{\chi}}\rfloor \text{ and } z_j =\frac{b-1}{2}-j -\frac{s_{\chi}(b_j-1)}{2}
\]
for $j=0,\cdots, \min\{s_{\chi},b\}-1$.
\begin{lemma}\label{product lem}
    There is an isomorphism
    \begin{equation}\label{prod eq}
            u(\chi;a,b)\simeq \prod_{j=0}^{\min\{s_{\chi},b\}-1} |\Nrd|^{z_j}_F \, U(\chi;a,b_j),
    \end{equation}
    where $|z_j|<\frac{s_{\chi}}{2}$ for each $0\leq j\leq \min\{s_{\chi},b\}-1$.
\end{lemma}
\begin{proof}
    For $z_j$, we organize it as
    \[
    z_j=\frac{b-1-j}{2} -\frac{s_{\chi}}{2}\lfloor\frac{b-1-j}{s_{\chi}}\rfloor-\frac{j}{2}.
    \]
    Let $k$ be the integer such that
    \[
    ks_{\chi} \leq b-1-j< (k+1)s_{\chi}.
    \]
    Then
    \[
-\frac{j}{2}=\frac{k}{2}s_{\chi}- \frac{s_{\chi}}{2}\cdot k -\frac{j}{2}   \leq z_j < \frac{k+1}{2}s_{\chi}- \frac{s_{\chi}}{2}\cdot k -\frac{j}{2}= \frac{s_{\chi}-j}{2}.
    \]
    Then $|z_j|<\frac{s_{\chi}}{2}$ follows from the fact that $0\leq j\leq s_{\chi}-1$. It implies that the cuspidal line for $|\Nrd|^{z_j}_F \, U(\chi;a,b_j)$ is different for each $0\leq j\leq s_{\chi}-1$. Thus, the right hand side of~\eqref{prod eq} is irreducible. Then the isomorphism comes from the intertwining operator.
\end{proof}
We remark that when $\mathrm{D}=F$, $u(\chi;a,b)=U(\chi;a,b)$.
\begin{definition}\label{Arthur-type rep def}
    An irreducible representation $\pi$ is of Arthur-type if there exists a set of unitary cuspidal representations $\{\chi_i\}_{i=1}^k$ and positive integers $a_1,\cdots,a_k,b_1,\cdots,b_k$ such that
    \[
    \pi\simeq u(\chi_1;a_1,b_1)\times \cdots\times u(\chi_k;a_k,b_k).
    \]
\end{definition}
 By the classification of the unitary dual (see~\cite[Theorem 2.2]{Cai23} and~\cite[section 7]{BR10}), each $u(\chi_i;a_i,b_i)$ is unitary, and their product is irreducible. In~\cite{Ba07,Ba08}, a global Jacquet-Langlands correspondence of the discrete spectrum is established. It is compatible with the local transfer between Arthur-type representations. In particular, Arthur-type representation defined above is indeed the local component of some automorphic representation in the discrete spectrum.

\subsection{Classification and separation of exceptional representations}\label{class and separation exc sec}
We recall the $L$-factor of irreducible admissible representations of $G_n$ under Langlands classification. 
\begin{proposition}[ {\cite[Proposition 3.5, Proposition 5.11]{GJ72}}]\label{L-func prop}
    Suppose that 
    \[
    \pi=LQ(\St(\Delta_1)\times\cdots\times\St(\Delta_r)),
    \]
    where $\Delta_i=[x_i,y_i]_{\chi_i}$ and $\chi_i$ is a cuspidal representation, then
    \[
    L\left(s,\pi\right) = \prod_{i=1}^{r} L\left(s+s_{\chi_i}y_i+\frac{s_{\chi_i}-1}{2},\tau_i\right),
\]
where $\mathrm{JLC}(\chi_i)=\tau_i$ (for this notation, see subsection~\ref{Arthur-type sec}) and
\[
    L\left(s,\tau_i\right) = \begin{cases}
       \frac{1}{1-\tau_i\left(\varpi\right)q^{-s}} \quad & \textit{if $\tau_i$ is a unramified character of $\GL_1(F)$};\\[10pt]

       1 \quad & \textit{otherwise}. 
    \end{cases}
\]
\end{proposition}
Now we consider the essential Speh representation $\pi=|\Nrd|^zU(\chi;a,b)$, where $z\in \BC$, $a,b\in\BZ_{>0}$ and $\chi$ is a unitary cuspidal representation. By the Proposition, we have
\[
L(s,\pi)=\prod_{j=1}^b L\left(s+z-\frac{1}{2}+s_{\chi}(\frac{a+b+1}{2}-j),\rho\right),
\]
where $\rho=\mathrm{JLC}(\chi)$ is also unitary. Thus, $L(s,\pi)$ has a pole at
\[
s=s_0=\frac{d(m-n)+1}{2}
\]
if and only if $\chi$ is the trivial character of $G_1$ and 
\begin{equation}\label{pole condition eq}
     s_0+z-\frac{1}{2}+s_{\chi}(\frac{a+b+1}{2})\in \{s_{\chi}, 2s_{\chi},\cdots, bs_{\chi}\}.
\end{equation}
Furthermore, $\ord_{s=s_0}L(s,\pi)\leq 1$.
In conclusion, if $\ord_{s=s_0}L(s,\pi)= 1$, then $s_{\chi}=d$, and \eqref{pole condition eq} is equivalent to
 \begin{equation}\label{pole condition eq useful}
     \frac{m-n}{2} +\frac{z}{d}\in\left\{ \frac{1-(a+b)}{2},\frac{3-(a+b)}{2},\cdots,\frac{b-a-1}{2} \right\}.
 \end{equation}
\begin{lemma}\label{L-function holo lem}
    Suppose $\pi$ is an Arthur-type representation of $G_n$. If $L(s,\pi^{\vee})$ is holomorphic at $s=s_0$, then $L(s,\pi)$ is also holomorphic at $s=s_0$.
\end{lemma}
\begin{proof}
    By Lemma~\ref{product lem}, $\pi$ is a product of the following two classes of representations:
    \begin{enumerate}
        \item $\lambda_1:=|\Nrd|^{-z}U(\chi;a,b)\times |\Nrd|^zU(\chi;a,b)$ for some unitary cuspidal representation $\chi$ and positive integer numbers $a,b$;
        \item $\lambda_2:= U(\chi;a,b)$ for some unitary cuspidal representation $\chi$ and positive integer numbers $a,b$;
    \end{enumerate}
    Suppose that $L(s,\lambda_1)$ has a pole at $s=s_0$. Then $\chi$ is the trivial character of $G_1$. Thus, 
    \[
    \lambda_1^{\vee}\simeq |\Nrd|^{z}U(\chi^{\vee};a,b)\times |\Nrd|^{-z}U(\chi^{\vee};a,b)\simeq \lambda_1.
    \]
    Thus, $L(s,\lambda_1^{\vee})$ has a pole at $s=s_0$, which is absurd. Similarly, $L(s,\lambda_2)$ should also be holomorphic at $s=s_0$. Thus, $L(s,\pi)$ is holomorphic at $s=s_0$.
\end{proof}

\begin{theorem}\label{separation thm}
   Let $n,m$ be positive integers. Let $\beta\in\Irr(G_l)$ where $0\leq l<n$. Let $\alpha\in\Irr(G_{n-l})$ be an Arthur-type representation that is not exceptional with respect to $(n,m)$. Then
   \begin{enumerate}
       \item If $m<n-l$, then $\Ext^i_{G_n}(\omega_{n,m},\alpha\times \beta)=0$ for any integer $i$;
        \item If $m\geq n-l$, then
    \end{enumerate}
    \[
    \Ext^i_{G_n}(\omega_{n,m},\alpha\times \beta)_{\sm}\simeq \alpha\times \Ext^i_{G_{l}}(\omega_{l,m-n+l},\beta)_{\sm},
    \]
    where $\omega_{l,m-n+l}$ is the trivial representation when $l=0$ or $m-n+l=0$.
\end{theorem}
\begin{proof}
    By Lemma~\ref{L-function holo lem}, we can assume that $L(s,\alpha^{\vee})$ is holomorphic at $s=s_0$. Write $\alpha$ as a product of essential Speh representations
    \[
    \alpha=\prod_{j=1}^r |\Nrd|^{z_j} U(\chi_j;a_j,b_j).
    \]
    Arguing by induction on $r$, we can assume $\alpha=|\Nrd|^{z} U(\chi;a,b)$ such that $z\in \BR$, $a,b\in\BZ_{>0}$ and $\chi$ is a unitary cuspidal representation. 

   By second adjointness,
   \[
   \Ext^i_{G_n}(\omega_{n,m}, \alpha\times\beta)_{\sm}\simeq \Ext^i_{L_{n-l,l}}(r_{n-l,l}(\omega_{n,m}),\alpha\boxtimes\beta)_{\sm}.
   \]
   By Proposition~\ref{Kudla filtration left}, $r_{n-l,l}(\omega_{n,m})$ has a decreasing filtration
   \begin{equation}\label{filtration right kudla in separate eq}
       0=\widetilde{J}_{k+1}\subset \widetilde{J}_k\subset \cdots\subset \widetilde{J}_0= r^{G_n}_{n-l,l}(\omega_{n,m}),
   \end{equation}
  where $k=\min\{n-l,m\}$, such that as representations of $L_{n-t,t}\times G_m'$, 
   \[
   \widetilde{\tau}_j\simeq \widetilde{J}_j/\widetilde{J}_{j+1}\simeq
\ind_{P_{n-l-j,j}\times G_{l}\times P'_{j,m-j}}^{L_{n-l,l}\times G'_m}
\bigl(\widetilde{\xi}_{n-l,j}\otimes\rho_j\otimes\omega_{l,m-j}\bigr),
\]
where $\rho_j$ is defined by (2.2), and $\widetilde{\xi}_{n-l,j}$ is the character
\[
\widetilde{\xi}_{n-l,j}=
\begin{cases}
\nu^{\frac{m-l-j}{2}}&\text{on }G_{n-l-j},\\
\nu^{\frac{m+n-2l-j}{2}}&\text{on }G_j,\\
\nu^{\frac{n-l-j}{2}}&\text{on }G_{l},\\
\nu'^{\frac{-m-n+2l+j}{2}}&\text{on }G'_j,\\
\nu'^{\frac{-n+l+j}{2}}&\text{on }G'_{m-j}.
\end{cases}
\]
When $0\leq j<k$, we have
\begin{align*}
    &\Ext^i_{L_{n-l,l}}(\widetilde{\tau}_j, \alpha\boxtimes\beta)_{\sm}\\
    \simeq &\ind_{P'_{j,m-j}}^{ G'_m} \left( \Ext^i_{L_{n-l,l}}(\ind_{P_{n-l-j,j}\times G_{l}}^{L_{n-l,l}}\bigl(\widetilde{\xi}_{n-l,j}\otimes\rho_j\otimes\omega_{l,m-j}\bigr)_{\sm}, \alpha\boxtimes\beta) \right),
\end{align*}
where K\"unneth formula~\cite[Theorem 3.5]{Pr23} identifies
\begin{align*}
   & \Ext^i_{L_{n-l,l}}\left(\ind_{P_{n-l-j,j}\times G_{l}}^{L_{n-l,l}}\bigl(\widetilde{\xi}_{n-l,j}\otimes\rho_j\otimes\omega_{l,m-j}\bigr)_{\sm}, \alpha\boxtimes\beta \right)_{\sm}\\
    \simeq &\bigoplus_{i=c+d}\Ext^c_{G_{n-l}}\left(\ind_{P_{n-l-j,j}}^{G_{n-l}}(\widetilde{\xi}_{n-l,j}\otimes\rho_j),\alpha\right)_{\sm}\boxtimes \Ext^d_{G_l}(\omega_{l,m-j},\beta \nu^{\frac{-n+l+j}{2}} )_{\sm} \nu'^{\frac{n-l-j}{2}} .
\end{align*}
We study the first term in the K\"unneth formula. By second adjointness,
\[
\Ext^c_{G_{n-l}}\left(\ind_{P_{n-l-j,j}}^{G_{n-l}}(\widetilde{\xi}_{n-l,j}\otimes\rho_j),\alpha\right)_{\sm}\simeq \Ext^c_{L_{n-l-j,j}}\left(\widetilde{\xi}_{n-l,j}\otimes\rho_j,\ov{r}_{n-l-j,j}(\alpha)\right)_{\sm},
\]
where the character $\widetilde{\xi}_{n-l,j}$ on $G_{n-l-j}$ has cuspidal supports
\begin{equation}\label{cuspidal support of character eq}
    \left\{ \nu_1^{\frac{m-n+1}{2}}, \nu_1^{\frac{m-n+3}{2}},\cdots, \nu_1^{\frac{m+n-1}{2}-l-j} \right\}.
\end{equation}
Thus, by comparing this cuspidal support with the cuspidal support of the first factor of $\ov{r}_{n-l-j,j}(\alpha)$, we find that
\begin{equation}
    \Ext^c_{L_{n-l-j,j}}\left(\widetilde{\xi}_{n-l,j}\otimes\rho_j,\ov{r}_{n-l-j,j}(\alpha)\right)=0 
\end{equation}
unless $\chi$ is the trivial character of $G_1$. Assume $\chi=\mathbf{1}_1$. By Lemma~\ref{Speh Jacquet lem}, the $G_{n-l-j}$-factor of the irreducible component of $\ov{r}_{n-l-j,j}(\alpha)$ has form
\[
|\Nrd|^z  LQ\left(\St([\frac{b-a}{2},c_1])\times \cdots\times \St([1-\frac{a+b}{2},c_b])\right)
\]
for some $c_1>\cdots>c_b$ such that 
\[
\sum_{i=1}^b (c_i-\frac{b-a}{2}+i)= n-l-j.
\]
Suppose it has a same cuspidal support as~\eqref{cuspidal support of character eq}, then the left endpoint of cuspidal supports coincide. Thus,
\[
   \frac{z}{d}+1-\frac{a+b}{2} \leq\frac{m-n+1}{2}\leq \frac{b-a}{2}+\frac{z}{d}
\]
and $\frac{b-a}{2}+\frac{z}{d}-\frac{m-n+1}{2}\in \BZ$. In other words,
\[
 \frac{m-n}{2} -\frac{z}{d}\in\left\{ \frac{1-(a+b)}{2},\frac{3-(a+b)}{2},\cdots,\frac{b-a-1}{2} \right\},
\]
which contradicts to the fact that $L(s,\alpha^{\vee})$ is holomorphic. In conclusion, the first term in the K\"unneth formula
\[
\Ext^c_{G_{n-l}}\left(\ind_{P_{n-l-j,j}}^{G_{n-l}}(\widetilde{\xi}_{n-l,j}\otimes\rho_j),\alpha\right)_{\sm}=0,
\]
for any integer $c$ when $0\leq j<k$, and thus $\Ext^i_{L_{n-l,l}}(\widetilde{\tau}_j, \alpha\boxtimes\beta)_{\sm}=0$ for any integer $i$ when $0\leq j<k$. This proves statement (1).

Now, we assume $m\geq n-l$. Therefore, $k=n-l$. At this time, when $j=k$, the first term in the K\"unneth formula is 
\[
\nu'^{\frac{m-l}{2}}\Ext^c_{G_{n-l}}(\rho_{n-l}\nu^{\frac{m-l}{2}},\alpha)_{\sm}\simeq \nu'^{\frac{m-l}{2}}\Ext^c_{G_{n-l}}(\rho_{n-l},\nu^{\frac{l-m}{2}}\alpha)_{\sm}\simeq \begin{cases}
    \alpha & c=0\\
    0 & c>0,
\end{cases}
\]
since $\rho_{n-l}$ is projective as a representation of $G_{n-l}$. Thus,
\[
\Ext^i_{L_{n-l,l}}(\widetilde{\tau}_k, \alpha\boxtimes\beta)_{\sm}\simeq \alpha \boxtimes\Ext^i_{G_l}(\omega_{l,m-k},\beta)_{\sm} ,
\]
where $\omega_{l,m-k}$ is the trivial representation when $l=0$ or $m-k=0$. This completes the proof by the long exact sequence associated to the filtration~\eqref{filtration right kudla in separate eq}.
\end{proof}

\subsection{Big theta for exceptional Speh representations}\label{computation big theta speh sec}
In this subsection, we consider the big theta for exceptional essential Speh representations appearing in the $u(\chi;a,c)$ for some unitary cuspidal representation $\chi$ and positive integers $a,c$.

Let $\pi=|\Nrd|^zU(\chi;a,b)$, where $b\in\BZ_{>0}$ and $z\in\BR$ such that $|z|<\frac{s_{\chi}}{2}$. Then by~\eqref{pole condition eq}, $\pi$ is exceptional if and only if $\chi=\mathbf{1}_1$, $z=0$, and
\begin{equation}
     \frac{m-n}{2} \in\left\{ \frac{1-(a+b)}{2},\frac{3-(a+b)}{2},\cdots,\frac{b-a-1}{2} \right\}.
\end{equation}
In other words,
\begin{equation}\label{definition of number h}
     h:= \frac{b-a-(m-n)+1}{2}\in \BZ \text{ and } 1\leq h\leq b.
\end{equation}
The next lemma shows that if $u(\chi;a,c)$ is exceptional, the product expression of it~\eqref{prod eq} contains exactly one exceptional Speh factor.
\begin{lemma}\label{one exc Speh factor lem}
    Let $(n,m)$ be a pair of positive integers. Then 
    \[
    \ord_{s=s_0} L(s, u(\chi;a,c))\leq 1.
    \]
\end{lemma}
\begin{proof}
   By Proposition~\ref{L-func prop}, we have
\begin{equation}\label{product factor of L-function eq}
    L(s,\pi) = \prod_{i=1}^c L\left(s+\frac{c}{2}-i +\frac{as_{\chi}}{2},\rho\right),
\end{equation}
where $\rho=\mathrm{JLC}(\chi)$ is unitary. Since the factor 
\[
L\left(s+\frac{c}{2}-i +\frac{as_{\chi}}{2},\rho\right)
\]
has a pole at $s=s_0$ only if $\rho$ is the trivial character of $\GL_1(F)$ and $s_0+\frac{c}{2}-i +\frac{ad}{2}=0$, at most one factor in~\eqref{product factor of L-function eq} can have a pole at $s=s_0$.
\end{proof}

We formulate the theorem concerning the big theta of exceptional Speh representations.
\begin{theorem}\label{big theta speh thm}
  Let $n,m$ be two positive integers. Let $\pi=U(\mathbf{1}_1;a,b)$ be an exceptional Speh representation of $G_n$ with respect to $(n,m)$. Then
    \begin{equation}
        \Theta_{n,m}(\pi)\simeq\begin{cases}
            0 &\text{if } a>1,m<n  ;\\
            \mathbf{1}_{m-n}\times \pi &\text{if } a>1,m\geq n;\\
       \nu^{\frac{2n-m}{4}}\mathbf{1}_{\frac{m}{2}}\times  \nu^{-\frac{2n-m}{4}}\mathbf{1}_{\frac{m}{2}} & \text{if } a=1 \text{ and } m \text{ is even}.
        \end{cases}
    \end{equation}
\end{theorem}
\begin{remark}\label{speh theta rem}
   This Theorem can be extended to the exceptional representation $\pi=u(\mathbf{1}_1;a,c)$ with respect to $(n_0=ac,m_0)$. By Lemma~\ref{product lem}, $\pi$ is exceptional if and only if there exists $0\leq j<\min\{d,c\}$, such that $z_j =\frac{c-1}{2}-j -\frac{d(c_j-1)}{2}=0$,
  \begin{equation}
    h=\frac{c_j-a-(m_0-n_0)+1}{2}\in \BZ \text{ and } 1\leq h\leq c_j,
\end{equation}
where $c_j=1+\lfloor\frac{c-1-j}{d}\rfloor$.  Moreover, Lemma~\ref{one exc Speh factor lem} ensures that such $j$ is unique. Let $b=c_j$. Then
\[
\pi_e= U(\mathbf{1}_1;a,b).
\]
Thus, thanks to Proposition~\ref{separation prop intro}, $\Theta_{n_0,m_0}(\pi)\neq 0$ only if $n_0-m_0\geq ab$. Then, one can apply Theorem~\ref{big theta speh thm} to $\pi_e$ with $n=ab$ and $m=m_0-n_0+ab$. In particular, when $a>1$ and $m_0\geq n_0$, we have
\[
\Theta_{n_0,m_0}(\pi)\simeq \mathbf{1}_{m_0-n_0}\times \pi.
\]
\end{remark}

\medskip

 Before proving the full theorem, we establish the necessary groundwork and address the case $a>1$. We defer the case $a=1$ to the next subsection, as it requires a distinct approach.

We compute through the rank filtration. By Lemma~\ref{rank filtration}, there is a decreasing filtration of $\omega_{n,m}$:
\[
0=\widetilde{R}_{e+1}\subset \widetilde{R}_e \subset \cdots\subset \widetilde{R}_0=\omega_{n,m},
\]
where $e=\min{\{n,m\}}$, such that as representations of $G_n\times G_m'$, 
    \begin{equation}\label{unitary normalized rank filtration eq}
         \widetilde{R}_k/\widetilde{R}_{k+1}\simeq \ind_{P_k\times \ov{P'_k}}^{G_n\times G_m'}(\rho_k \otimes \widetilde{\eta}_k),
    \end{equation}
    where $\widetilde{\eta}_k$ is a character of the standard Levi subgroup of $P_k\times \ov{P'_k}$, given by
    \begin{equation}
      \widetilde{\eta}_k=  \begin{cases}
             \nu^{\frac{k-n-m}{2}} & \text{ on } G_k\\
             \nu^{\frac{k-m}{2}} & \text{ on } G_{n-k}\\
             \nu'^{\frac{m-k+n}{2}}& \text{ on } G_k' \\
             \nu'^{\frac{n-k}{2}} & \text{ on } G_{m-k}'
        \end{cases}.
    \end{equation}
    We first compute the contribution of each piece in the rank filtration when $m\geq n$.
    \begin{lemma}\label{graded piece computation lem}
         Let $m\geq n$ be two positive integers. Let $\pi=U(\mathbf{1}_1;a,b)$ be an exceptional Speh representation of $G_n$ with respect to $(n,m)$. Then
         \[
         \Hom_{G_n}( \widetilde{R}_k/\widetilde{R}_{k+1}, \pi)_{\sm}\simeq\begin{cases}
         \mathbf{1}_{m-n}\times \pi   &\text{ when }k=n;\\
        \nu_{m-k}^{\frac{k-n}{2}}\times A    &\text{ when } k=n-h;\\
         0   &\text{ otherwise } ,
         \end{cases}
         \]
         where $A=Z(\sum_{i=1}^{a-1}[\frac{a-b}{2}-i,\frac{a+b}{2}-1-i] +[\frac{1-m+n}{2},\frac{a+b}{2}-1])$ and $h$ is defined by~\eqref{definition of number h}.
    \end{lemma}
    \begin{proof}
        By~\eqref{unitary normalized rank filtration eq}, we have
        \begin{align*}
      &   \Hom_{G_n}( \widetilde{R}_k/\widetilde{R}_{k+1}, \pi)_{\sm}\simeq    \Hom_{G_n}\left(   \ind_{P_k\times \ov{P'_k}}^{G_n\times G_m'}(\rho_k \otimes \widetilde{\eta}_k),\pi\right)_{\sm}\\
      &\simeq \ind_{\ov{P'_k}}^{G_m}\left( \Hom_{L_{k,n-k}}(\rho_k\otimes\widetilde{\eta}_k,\ov{r}_{k,n-k}(\pi))\right)_{\sm}.
        \end{align*}
        When $k<n$, then the character $\widetilde{\eta}_k$ on $G_{n-k}$ is $\nu^{\frac{k-m}{2}}$, whose cuspidal support is
        \[
        \left\{ \nu_1^{\frac{2k-m-n+1}{2}}, \nu_1^{\frac{2k-m-n+3}{2}},\cdots, \nu_1^{\frac{n-m-1}{2}} \right\}.
        \]
        On the other hand, by Lemma~\ref{Speh Jacquet lem}, the irreducible component of the $G_{n-k}$-factor of $\ov{r}_{k,n-k}(\pi)$  is
        \begin{equation}\label{irreducible factor eq}
             LQ\left(\St([c_1,\frac{a+b}{2}-1])\times \cdots\times\St([c_b,\frac{a-b}{2}])\right)
        \end{equation}
        for some $c_1>c_2>\cdots>c_b$ such that
        \[
        \sum_{i=1}^b(\frac{a+b}{2}-c_i+1-i)=n-k.
        \]
        It is a character if and only if $c_i\geq \frac{a+b}{2}-i$ for each $i$. Observe that if $c_i>\frac{a+b}{2}-i$, then $c_j>\frac{a+b}{2}-j$ for any $j<i$. Therefore, if~\eqref{irreducible factor eq} is a character, then it is isomorphic to
        \begin{equation}\label{character component eq}
         B := LQ\left(\nu_1^{\frac{a-b}{2}+(n-k)-1}\times\cdots\times \nu_1^{\frac{a-b}{2}}\right).
        \end{equation}
        Consequently, applying Lemma~\ref{Speh Jacquet lem} and the M\oe glin-Waldspurger's algorithm to the $G_{k}$-factor corresponding to $B$, we obtain an isomorphism as representations of $G_k$:
        \[
        \Hom_{G_{n-k}}(\widetilde{\eta}_k,\ov{r}_{k,n-k}(\pi))\simeq \begin{cases}
            A  & \text{ when }\frac{2k-m-n+1}{2}= \frac{a-b}{2} \Leftrightarrow k=n-h;\\
            0 & \text{ otherwise}.
        \end{cases}
        \]
        Thus, when $k<n$,
        \[
        \ind_{\ov{P'_k}}^{G_m}\left( \Hom_{L_{k,n-k}}(\rho_k\otimes\widetilde{\eta}_k,\ov{r}_{k,n-k}(\pi))_{\sm}\right)\simeq \begin{cases}
              \nu_{m-k}^{\frac{k-n}{2}}\times A    &\text{ when } k=n-h;\\
         0   &\text{ otherwise }.
         \end{cases}
        \]
       The case $k=n$ follows from an argument analogous to the one in Lemma~\ref{lift off boundary lemma}.
    \end{proof}
    \begin{remark}
        When $k\neq n$ and $k\neq n-h$, $\Ext^i_{G_n}( \widetilde{R}_k/\widetilde{R}_{k+1}, \pi)_{\sm}$ might not be zero when $i>0$. This is an intrinsic difference comparing to the computation of the character in~\cite[Lemma 5.1]{CLLTZ25} when $a>1$.
    \end{remark}
 We glue the information of each graded piece through the multiplicity-one Theorem~\ref{multi one intro thm} and Godement-Jacquet zeta integral.
    \begin{lemma}\label{glue embedding lem}
         Let $m\geq n$ be two positive integers. Let $\pi=U(\mathbf{1}_1;a,b)$ be an exceptional Speh representation with respect to $(n,m)$. Then there is an injection
         \[
          \Hom_{G_n}(\omega_{n,m}, \pi)_{\sm}\hookrightarrow \nu_{m-n+h}^{-\frac{h}{2}}\times A .
         \]
    \end{lemma}
    \begin{proof}
        The short exact sequence
        \[
        0\lra \widetilde{R}_n \lra \omega_{n,m} \lra \omega_{n,m}/\widetilde{R}_n\lra 0
        \]
        induces the exact sequence
        \[
       0\lra \Hom_{G_n}(\omega_{n,m}/\widetilde{R}_n, \pi)_{\sm}\lra \Hom_{G_n}(\omega_{n,m}, \pi)_{\sm} \stackrel{\CR}{\lra} \Hom_{G_n}(\widetilde{R}_n, \pi)_{\sm},
        \]
        where $\Hom_{G_n}(\widetilde{R}_n, \pi)\simeq \mathbf{1}_{m-n}\times \pi$. Since $\pi$ is exceptional, Lemma~\ref{GJ integral on boundary lem} shows that
        \[
     (\widetilde{R}_n\otimes \pi^{\vee})_{G_n} \stackrel{\CR^{\vee}}{\lra} (\omega_{n,m}\otimes \pi^{\vee})_{G_n} \stackrel{Z^{\sharp}}{\lra} \mathbf{1}_{m-n}\times \pi^{\vee}
        \]
        is a complex, namely $Z^{\sharp}\circ\CR^{\vee}=0$. Since $\mathbf{1}_{m-n}=U(\mathbf{1}_1;1,m-n)$ is also a Speh representation, $\mathbf{1}_{m-n}\times \pi^{\vee}$ is irreducible. Thus, $Z^{\sharp}$ is surjective by Proposition~\ref{GJ proposition} (3), and
        \[
        \theta_{n,m}(\pi)\simeq \mathbf{1}_{m-n}\times \pi^{\vee}.
        \]
        Since $(\widetilde{R}_n\otimes \pi^{\vee})_{G_n}\simeq \mathbf{1}_{m-n}\times \pi^{\vee}$, by multiplicity-one Theorem~\ref{multiplicity one thm}, we know that $\CR^{\vee}=0$ and $\CR=0$. In other words,
        \[
         \Hom_{G_n}(\omega_{n,m}/\widetilde{R}_n, \pi)_{\sm}\simeq \Hom_{G_n}(\omega_{n,m}, \pi)_{\sm}.
        \]
        Applying the exact sequence successively,
        \[
          0\lra \Hom_{G_n}(\omega_{n,m}/\widetilde{R}_k, \pi)_{\sm}\lra \Hom_{G_n}(\omega_{n,m}/\widetilde{R}_{k+1}, \pi)_{\sm} \stackrel{\CR}{\lra} \Hom_{G_n}(\widetilde{R}_k/\widetilde{R}_{k+1}, \pi)_{\sm},
        \]
        we obtain
        \[
          \Hom_{G_n}(\omega_{n,m}/\widetilde{R}_n, \pi)_{\sm}\simeq   \Hom_{G_n}(\omega_{n,m}/\widetilde{R}_{n-h+1}, \pi)_{\sm}    
        \]
      and $\Hom_{G_n}(\omega_{n,m}/\widetilde{R}_{n-h}, \pi)_{\sm}=0$. Consequently,
       \[
       \Hom_{G_n}(\omega_{n,m}, \pi)_{\sm}\hookrightarrow \Hom_{G_n}(\widetilde{R}_{n-h}/\widetilde{R}_{n-h+1}, \pi)_{\sm}\simeq \nu_{m-n+h}^{-\frac{h}{2}}\times A.
       \]
    \end{proof}
    Then, we apply the same argument as Lemma~\ref{graded piece computation lem} and Lemma~\ref{glue embedding lem} to the rank filtration on the Fourier dual side. That is Kudla's filtration in~\eqref{filtration right kudla in separate eq} when $t=n$. 

    The following lemma corresponds to Lemma~\ref{graded piece computation lem}. 
    \begin{lemma}\label{graded piece computation Fourier lem}
         Let $m\geq n$ be two positive integers. Let $\pi=U(\mathbf{1}_1;a,b)$ be an exceptional Speh representation of $G_n$ with respect to $(n,m)$. Then
         \[
         \Hom_{G_n}( \widetilde{J}_k/\widetilde{J}_{k+1}, \pi)_{\sm}\simeq\begin{cases}
          \pi \times \mathbf{1}_{m-n}   &\text{ when }k=n;\\
        A^{\vee}  \times \nu_{m-k}^{\frac{n-k}{2}}  &\text{ when }  k=n-h;\\
         0   &\text{ otherwise } ,
         \end{cases}
         \]
         where $A=Z(\sum_{i=1}^{a-1}[\frac{a-b}{2}-i,\frac{a+b}{2}-1-i] +[\frac{1-m+n}{2},\frac{a+b}{2}-1])$.
    \end{lemma}
    \begin{proof}
        The proof is similar to that of Lemma~\ref{graded piece computation lem} by replacing the decreasing filtration $\widetilde{R}_i$ by the decreasing filtration $\widetilde{J}_i$. We leave it for interested readers.
    \end{proof}

    We will also need the integral representation for the $L$-function of the contragredient representation, which is also related to the theta lift of $\pi$. Given $\phi\in S_{n,m}$, $v\in \pi$, $\lambda\in \pi^{\vee}$, and $s\in \BC$, we can associate a contragredient zeta integral:
\[
\widehat{Z}(\phi, \lambda, v;s):=\int_{G_n} \CF(\phi)\begin{pmatrix}
    g \\
    0_{m-n,n} 
\end{pmatrix} \lambda(g^{-1}\cdot v)|\Nrd(g)|_{F}^{s+\frac{dn-1}{2}}\, dg,
\]
where $\CF(\phi)$ is the Fourier transform defined in~\eqref{Fourier transform eq}. Similar as Proposition~\ref{GJ proposition}, we have:
\begin{itemize}
    \item The contragredient zeta integral $\widehat{Z}(\phi,\lambda,v;s)$ is absolutely convergent when $\mathrm{Re} s$ is sufficiently large, and it admits a meromorphic continuation to the whole complex plane, which is also denoted by $\widehat{Z}(\phi,\lambda,v;s)$.  
        \item 
        \[
         \widehat{Z}^{\sharp}(\phi,\lambda,v;s):=\frac{\widehat{Z}(\phi,\lambda,v;s)}{L(s,\pi^{\vee})}
        \]
        is a holomorphic function since $\lambda(g^{-1}\cdot v)$ is a matrix coefficient of $\pi^{\vee}$ and 
        \[
        S_{n,m}\stackrel{\CF}{\lra} S_{m,n}\stackrel{  \mathrm{res}   }{\lra} S_{n,n}
        \]
        is surjective, where $\mathrm{res}$ is the restriction to square matrices.
        \item It induces a non-zero $G_n\times P_n'$-equivariant map $\widehat{Z}^{\sharp}(\cdot,\cdot,\cdot;s): \omega_{n,m} \lra \pi\boxtimes \pi^{\vee}\otimes \xi_s'$, where 
       the action of $P_n'$ factors through $P_n'\lra G_n'\times G_{m-n}'$ and
       \begin{equation}
           \xi_s'=\begin{cases}
               |\Nrd|^{-s-\frac{d(n-m)-1}{2}}  & \text{ on } G_n\\
               |\Nrd|^{s-\frac{1}{2}}    & \text{ on } G_n'\\
               \nu^{-\frac{n}{2}}       & \text{ on } G_{m-n}'.
           \end{cases}
       \end{equation}
       Here, on the right hand side, $G_n$ acts on $\pi$ and $G_n'$ acts on $\pi^{\vee}$.
\end{itemize}
By Frobenius reciprocity, we obtain a non-zero element in
\[
 \widehat{Z}^{\sharp}\in \Hom_{G_n\times G_m'}(\omega_{n,m}\otimes \pi^{\vee}, \pi^{\vee}\times \mathbf{1}_{m-n})
\]
when specializing to $s=s_0:=\frac{1+d(m-n)}{2}$. 

Let $\widetilde{J}_n$ be the subrepresentation of $\omega_{n,m}$ defined in~\eqref{filtration right kudla in separate eq} and Remark~\ref{rank filtration dual rem}.
\begin{lemma}\label{daul GJ integral on boundary lem}
    If $L(s,\pi^{\vee})$ is not holomorphic at $s=s_0$, then $\widehat{Z}^{\sharp}$ restricted to $\widetilde{J}_n\otimes \pi^{\vee}$ is zero.
\end{lemma}
\begin{proof}
    By definition in Remark~\ref{rank filtration dual rem}, if $\phi\in\widetilde{J}_n$, then $\CF(\phi)$ is supported on the full rank locus of $M_{m,n}(\mathrm{D})$. Therefore,
    \[
     \int_{G_n} \CF(\phi)\begin{pmatrix}
    g \\
    0_{m-n,n} 
\end{pmatrix} \lambda(g^{-1}\cdot v)|\Nrd(g)|_{F}^{s+\frac{dn-1}{2}}\, dg
    \]
    is absolutely convergent for any $\lambda\in\pi^{\vee}, v\in\pi$ and $s\in \BC$ when $\phi\in\widetilde{J}_n$. Thus, it equals $\widehat{Z}(\phi,\lambda,v;s)$ for any $s\in \BC$. When $L(s,\pi^{\vee})$ is not holomorphic at $s=s_0$, we have $\widehat{Z}^{\sharp}(\phi,\lambda,v;s_0)=0$ for any $v\in\pi$ and $\lambda\in\pi^{\vee}$.
\end{proof}
Then, by a similar argument as Lemma~\ref{glue embedding lem}, we obtain the following result. We omit the proof.
 \begin{lemma}\label{dual glue embedding lem}
         Let $m\geq n$ be two positive integers. Let $\pi=U(\mathbf{1}_1;a,b)$ be an exceptional Speh representation with respect to $(n,m)$. Then there is an injection
         \[
          \Hom_{G_n}(\omega_{n,m}, \pi)_{\sm}\hookrightarrow A^{\vee}\times \nu_{m-n+h}^{\frac{h}{2}} .
         \]
    \end{lemma}

Now, we apply Lemma~\ref{glue embedding lem} and Lemma~\ref{dual glue embedding lem} to prove Theorem~\ref{big theta speh thm} when \(a>1\). We need to investigate the induced representation $\nu_{m-n+h}^{-\frac{h}{2}}\times A$ more closely. Let $\Delta_1$ be the Zelevinsky segment of $\nu_{m-n+h}^{-\frac{h}{2}}$, namely, $\nu_{m-n+h}^{-\frac{h}{2}}= Z(\Delta_1)$. And let 
\[
\Delta_2=[\frac{1-m+n}{2},\frac{a+b}{2}-1] , \Lambda_i= [\frac{a-b}{2}-i,\frac{a+b}{2}-1-i]
\]
for $1\leq i\leq a-1$ be Zelevinsky segments of $A$. Since $1\leq h\leq b$, by comparing the right endpoints of $\Delta_1$ and $\Lambda_{a-1}$, we have
\[
e(\Lambda_{a-1})-e(\Delta_1)=\frac{b-a}{2}-\frac{m-n-1}{2}=h \geq 1,
\]
which ensures that $\Delta_1$ is strictly contained in $\Lambda_i$ for each $i$. The Figure~\ref{fig:segments-exceptional-speh} shows the relative position of segments in the induced representation.
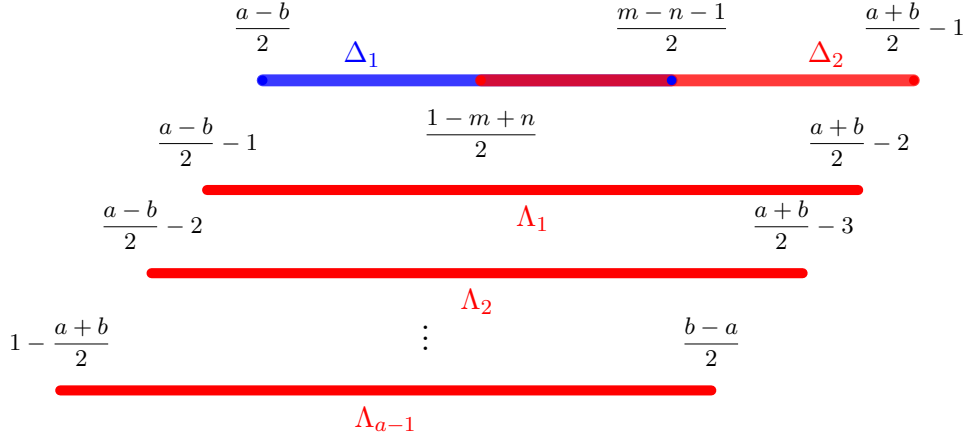
\begin{figure}[htbp]
\centering
\begin{adjustbox}{max width=0.98\textwidth}
\begin{tikzpicture}[
    x=1.05cm,
    y=1.00cm,
    endpoint/.style={circle,fill,inner sep=1.5pt},
    explabel/.style={font=\scriptsize},
    seglabel/.style={font=\small\bfseries}
]

%
%

\draw[
    blue,
    line width=4.5pt,
    line cap=round,
    opacity=.78
]
    (0,0)--(5.15,0);

\draw[
    red,
    line width=4.5pt,
    line cap=round,
    opacity=.78
]
    (2.75,0)--(8.20,0);

\fill[blue] (0,0) circle (1.8pt);
\fill[blue] (5.15,0) circle (1.8pt);

\fill[red] (2.75,0) circle (1.8pt);
\fill[red] (8.20,0) circle (1.8pt);

\node[explabel,above=7pt] at (0,0)
    {$\displaystyle \frac{a-b}{2}$};

\node[explabel,above=7pt] at (5.15,0)
    {$\displaystyle \frac{m-n-1}{2}$};

\node[explabel,below=7pt] at (2.75,0)
    {$\displaystyle \frac{1-m+n}{2}$};

\node[explabel,above=7pt] at (8.20,0)
    {$\displaystyle \frac{a+b}{2}-1$};

\node[seglabel,blue,above=1pt] at (1.25,0)
    {$\Delta_1$};

\node[seglabel,red,above=1pt] at (7.10,0)
    {$\Delta_2$};


\draw[
    red,
    line width=3.8pt,
    line cap=round
]
    (-0.70,-1.45)--(7.50,-1.45);

\fill[red] (-0.70,-1.45) circle (1.6pt);
\fill[red] (7.50,-1.45) circle (1.6pt);

\node[explabel,above=5pt] at (-0.70,-1.45)
    {$\displaystyle \frac{a-b}{2}-1$};

\node[explabel,above=5pt] at (7.50,-1.45)
    {$\displaystyle \frac{a+b}{2}-2$};

\node[seglabel,red,below=2pt] at (3.40,-1.45)
    {$\Lambda_1$};


\draw[
    red,
    line width=3.8pt,
    line cap=round
]
    (-1.40,-2.55)--(6.80,-2.55);

\fill[red] (-1.40,-2.55) circle (1.6pt);
\fill[red] (6.80,-2.55) circle (1.6pt);

\node[explabel,above=5pt] at (-1.40,-2.55)
    {$\displaystyle \frac{a-b}{2}-2$};

\node[explabel,above=5pt] at (6.80,-2.55)
    {$\displaystyle \frac{a+b}{2}-3$};

\node[seglabel,red,below=2pt] at (2.70,-2.55)
    {$\Lambda_2$};


\node[font=\large] at (2.05,-3.30)
    {$\vdots$};


\draw[
    red,
    line width=3.8pt,
    line cap=round
]
    (-2.55,-4.10)--(5.65,-4.10);

\fill[red] (-2.55,-4.10) circle (1.6pt);
\fill[red] (5.65,-4.10) circle (1.6pt);

\node[explabel,above=5pt] at (-2.55,-4.10)
    {$\displaystyle 1-\frac{a+b}{2}$};

\node[explabel,above=5pt] at (5.65,-4.10)
    {$\displaystyle \frac{b-a}{2}$};

\node[seglabel,red,below=2pt] at (1.55,-4.10)
    {$\Lambda_{a-1}$};

\end{tikzpicture}
\end{adjustbox}

\caption{
The segments in
$\nu_{m-n+h}^{-\frac{h}{2}}\times A$.
The segment $\Delta_1$ is blue, while
$\Delta_2,\Lambda_1,\ldots,\Lambda_{a-1}$,
which constitute the multisegment defining $A$, are red.
}
\label{fig:segments-exceptional-speh}
\end{figure}

\begin{proof}[proof of Theorem~\ref{big theta speh thm} when $a>1$]
    When $m\geq n$, since $\theta_{n,m}(\pi)\neq 0$, we have $\Hom_{G_n}(\omega_{n,m},\pi)_{\sm}\neq 0$. By the analysis about the Figure~\ref{fig:segments-exceptional-speh}, we know that $\Delta_1,\Delta_2$ and $\Lambda_i$ satisfy the assumption of Proposition~\ref{length two prop}. Hence, $\nu_{m-n+h}^{-\frac{h}{2}}\times A$ is of length two with socle $\mathbf{1}_{m-n}\times U(\mathbf{1}_{1};a,b)$ and cosocle
    \[
    Z(\Delta_1+\Delta_2+\sum_{i=1}^{a-1} \Lambda_i).
    \]
     Therefore, by Lemma~\ref{glue embedding lem}, there are two possibilities, namely, 
    \[
    \Hom_{G_n}(\omega_{n,m},\pi)_{\sm}\simeq \mathbf{1}_{m-n}\times U(\mathbf{1}_{1};a,b) \text{ or } \simeq \nu_{m-n+h}^{-\frac{h}{2}}\times A.
    \]
    Assume $ \Hom_{G_n}(\omega_{n,m},\pi)_{\sm}\simeq \nu_{m-n+h}^{-\frac{h}{2}}\times A$. In the Grothendieck group,
    \[
    [A^{\vee}\times \nu_{m-n+h}^{\frac{h}{2}}]= [(\nu_{m-n+h}^{-\frac{h}{2}}\times A)^{\vee}],
    \]
    which implies that $A^{\vee}\times \nu_{m-n+h}^{\frac{h}{2}}$ has length two and contains the irreducible component
    \[
    Z(\Delta_1+\Delta_2+\sum_{i=1}^{a-1} \Lambda_i)^{\vee}\simeq Z(\Delta_1^{\vee}+\Delta_2^{\vee}+\sum_{i=1}^a \Lambda_i^{\vee}).
    \]
    Since when $a>1$, we have $\ell(\Delta_1)\geq 1$ and $\ell(\Delta_2)\geq 1$, which implies
    \[
    \Delta_1^{\vee}+\Delta_2^{\vee}+\sum_{i=1}^a \Lambda_i^{\vee}\neq \Delta_1+\Delta_2+\sum_{i=1}^{a-1} \Lambda_i.
    \]
    This leads to a contradiction to Lemma~\ref{dual glue embedding lem}. Consequently, when $m\geq n$,
    \[
     \Hom_{G_n}(\omega_{n,m},\pi)_{\sm}\simeq \mathbf{1}_{m-n}\times U(\mathbf{1}_{1};a,b).
    \]
    When $m<n$ and $a>1$, then by the Langlands parameter of $\pi=U(\mathbf{1}_1;a,b)$ (see~\ref{multi-seg speh eq}), Howe duality Theorem~\ref{Howe duality thm} (2) tells us, there does not exists an irreducible representation $\sigma$ of $G_m$, such that
    \[
    \theta_{m,n}(\sigma)\simeq \pi.
    \]
    Thus, by Howe duality Theorem~\ref{Howe duality thm} (1), $\Hom_{G_n}(\omega_{n,m},\pi)=0$.
\end{proof}

\subsection{Proof of Theorem~\ref{big theta speh thm} when $a=1$}
When $a=1$, we have $n=b$. At this time, since $\pi$ is exceptional,
\[
h=\frac{b-a-(m-n)+1}{2}=\frac{2n-m}{2}\in \BZ,
\]
which implies that $m$ is an even integer. We also observe that $m-n+h=\frac{m}{2}$ and $A\simeq \nu^{\frac{h}{2}}_{\frac{m}{2}}$. Therefore, Lemma~\ref{dual glue embedding lem} will not provide additional information, and the argument in subsection~\ref{computation big theta speh sec} does not work.

Instead, we follows the strategy in~\cite{CLLTZ25} and argue by computing the extension group in the rank filtration. We will need the following lemma.
\begin{lemma}[\cite{BW00}, X. Proposition 2.6]\label{self-ext lem}
    Let $\chi$ be a character of $G_n$. Then
    \[
    \Ext^i_{G_n}(\chi,\chi)\simeq\begin{cases}
        \BC & \text{ when } i=0,1;\\
        0 & \text{ otherwise} .
    \end{cases}
    \]
\end{lemma}
The extension group in the rank filtration is given by the following lemma. The Hom-space is also computed in Lemma~\ref{graded piece computation lem}.
\begin{lemma}\label{extension group lem}
 Let $n,m$ be two integers.   Let $\pi=U(\mathbf{1}_{1};1,n)$ be the trivial representation of $G_n$ which is exceptional. Then
 \begin{enumerate}
     \item When $0\leq k<\min\{n,m\}$,
     \[
    \Ext^i_{G_n}(\widetilde{R}_k/\widetilde{R}_{k+1}, \pi)_{\sm}\simeq\begin{cases} \nu^{\frac{m-2n}{4}}\mathbf{1}_{\frac{m}{2}}\times  \nu^{\frac{2n-m}{4}}\mathbf{1}_{\frac{m}{2}} &\text{ if } k=\frac{m}{2} \text{ and }i=0,1;\\
    0 & \text{ otherwise}.
    \end{cases}
     \]
     \item When $k=\min\{n,m\}$,
     \[
    \Ext^i_{G_n}(\widetilde{R}_k/\widetilde{R}_{k+1}, \pi)_{\sm}\simeq\begin{cases} \mathbf{1}_{m-n}\times \mathbf{1}_n &\text{ if } n\leq m \text{ and } i=0;\\
    0 & \text{ otherwise}.
    \end{cases}
     \]
 \end{enumerate}
\end{lemma}
\begin{proof}
    Assuming Lemma~\ref{self-ext lem}, the proof is exactly the same as in~\cite[Lemma 5.1]{CLLTZ25}. We omit the details.
\end{proof}
\begin{proof}[Proof of Theorem~\ref{big theta speh thm} when $a=1$]
    By Lemma~\ref{extension group lem}, we have
    \[
 \Ext^i_{G_n}   \left(\omega_{n,m}/\widetilde{R}_{\frac{m}{2}+1},\pi\right)_{\sm}\simeq\Ext^i_{G_n}(\widetilde{R}_{\frac{m}{2}}/\widetilde{R}_{\frac{m}{2}+1}, \pi)_{\sm}
    \]
    On the other hand, using the long exact sequence associated to the decreasing filtration of $\widetilde{R}_{\frac{m}{2}+1}$, Lemma~\ref{extension group lem} implies that
    \[
    \Ext^i_{G_n}   \left(\widetilde{R}_{\frac{m}{2}+1},\pi\right)_{\sm}\simeq \begin{cases}
    \mathbf{1}_{m-n}\times \mathbf{1}_n  &\text{ if } n\leq m \text{ and } i=0; \\
     0 &\text{ otherwise}.
 \end{cases}
    \]
    for any integer $i$. Now, consider the short exact sequence
    \[
    0 \lra \widetilde{R}_{\frac{m}{2}+1}\lra \omega_{n,m} \lra \omega_{n,m}/\widetilde{R}_{\frac{m}{2}+1}\lra 0
    \]
    and its associated long exact sequence
    \begin{align*}
   0\lra  \Hom_{G_n}&(\omega_{n,m}/\widetilde{R}_{\frac{m}{2}+1},\pi)_{\sm}\lra \Hom_{G_n}(\omega_{n,m},\pi)_{\sm}\stackrel{\mathscr{T}}{\lra} \Hom_{G_n}(\widetilde{R}_{\frac{m}{2}+1},\pi)_{\sm}\\
   &\lra \Ext^1_{G_n}(\omega_{n,m}/\widetilde{R}_{\frac{m}{2}+1},\pi)_{\sm} \lra \Ext^1_{G_n}(\omega_{n,m},\pi)_{\sm}.
    \end{align*}
    Therefore, when $n>m$, 
\[
\Hom_{G_n}(\omega_{n,m}/\widetilde{R}_{\frac{m}{2}+1},\pi)_{\sm}\simeq \Hom_{G_n}(\omega_{n,m},\pi)_{\sm}\simeq \nu^{\frac{m-2n}{4}}\mathbf{1}_{\frac{m}{2}}\times  \nu^{\frac{2n-m}{4}}\mathbf{1}_{\frac{m}{2}}.
\]
When $n\leq m$, it suffices to show that $\mathscr{T}=0$. Since $\mathbf{1}_{m-n}\times \mathbf{1}_n$ is irreducible, by Godement-Jacquet zeta integral Proposition~\ref{GJ proposition} (3), 
\[
\theta_{n,m}(\pi)\simeq \mathbf{1}_{m-n}\times \mathbf{1}_n\simeq \theta_{n,m}(\pi)^{\vee}.
\]
If $\mathscr{T}\neq0$, then $\mathbf{1}_{m-n}\times \mathbf{1}_n$ is a quotient of $ \Hom_{G_n}(\omega_{n,m},\pi)_{\sm}$. Thus, $\mathscr{T}$ is a splitting of 
\[
\theta_{n,m}(\pi)^{\vee}\hookrightarrow  \Hom_{G_n}(\omega_{n,m},\pi)_{\sm},
\]
since $\theta_{n,m}(\pi)^{\vee}$ has multiplicity one in $\Hom_{G_n}(\omega_{n,m},\pi)_{\sm}$ by Theorem~\ref{multiplicity one thm},
which contradicts to the Howe duality Theorem~\ref{Howe duality thm}. This completes the proof.
\end{proof}

\section{A conjecture on irreducibility of big theta lift}
In this section, we discuss two special cases of the Conjecture~\ref{irr conj intro} mentioned in the introduction.
\subsection{Test representations and reducible big theta}\label{reducible sec}
This subsection is devoted to the proof of the following Proposition.
\begin{proposition}\label{reducible prop}
    Let $m=n$ be a positive integer. Assume $\mathrm{D}=F$. Let $\pi$ be an Arthur-type representation such that
    \[
   \pi\simeq U(\chi_1;a_1,b_1)\times \cdots\times U(\chi_k;a_k,b_k),
    \]
    where each Speh block $U(\chi_i;a_i,b_i)$ is exceptional with respect to $(n,m)$. If $a_i=1$ for some $i$, then $\Theta_{n,n}(\pi)$ is reducible.
\end{proposition}

We prove this Proposition by constructing a (GGP-)test representation for $\pi$. We first analyze the Langlands parameter of $\pi^-$. Since each Speh block is exceptional, the discussion in subsection~\ref{class and separation exc sec} shows that $\chi_i$ is the trivial character of $G_1$ for each $1\leq i \leq k$. Moreover, if $U(\mathbf{1}_1;1,b)$ is exceptional with respect to $(n,m)$, then
\[
 \frac{b-a-(m-n)+1}{2}=\frac{b}{2}\in \BZ.
\]
In other words, $b$ is even. The highest BZ-derivative formula with respect to Zelevinsky classification (Proposition~\ref{Zele class_properties}, part~\ref{part: highest derivative}) shows
\[
U(\mathbf{1}_1;1,b)^- =U(\mathbf{1}_1;1,b-1),
\]
whose Langlands parameter contains the trivial character of $G_1$. Thus, if we write
\[
\pi^-=LQ(\tau_1\times \cdots\times \tau_N),
\]
where $\tau_i$ is essentially square-integrable and $\tau_1\times \cdots\times \tau_N$ is a standard representation, then there exist $1\leq j\leq N$, such that $\tau_j=\mathbf{1}_1$. Let 
\[
\sigma=LQ(\tau_1\times \cdots\times \widehat{\tau_j}\times \cdots\times \tau_N),
\]
where, as usual, $\widehat{\cdot}$ over an argument means that the argument should be omitted. Then, $\sigma$ is an irreducible representation of $G_{n-d-1}$, where
\[
d= \sum_{i=1}^k a_i
\]
is the depth of $\pi$. Let $\rho$ be a cuspidal representation of $G_d$ which is assumed to be ramified when $d=1$. We claim that
\[
\lambda:=LQ(\tau_N^{\vee}\times \cdots\times \widehat{\tau_j^{\vee}}\times \cdots\times \tau_1^{\vee}\times \rho^{\vee})\in \Irr(G_{m-1})
\]
is a desired test representation. We first prove that the Hom-space for small theta is zero, which follows from a more general lemma.
\begin{lemma}\label{hom-zero lem}
    Let $\beta\in \Irr(G_n)$ and $\gamma\in \Irr(G_{n-1})$. Let $d$ be the depth of $\beta$. If there is a cuspidal representation $\rho \in \supp (\gamma)$ of $G_d$ such that $\rho\neq \xi\cdot \chi$ in $\Irr$ for any unramified character $\xi$ and cuspidal representation $\chi\in \supp(\beta)$, then
    \[
    \Hom_{G_{n-1}}(\beta,\gamma)=0.
    \]
\end{lemma}
\begin{proof}
    There is a decreasing filtration of $\beta|_{G_{n-1}}$ called the Bernstein-Zelevinsky filtration, whose subquotient is isomorphic to 
    \[
    Q_j:= \beta^{[j]} \times c-\ind_{U_{j-1}}^{G_{j-1}}(\psi_{j-1})
    \]
    for $1\leq j\leq d$. To prove the Lemma, it suffices to show
    \[
    \Hom_{G_{n-1}}(Q_j, \gamma)=0.
    \]
    By second adjointness, it is isomorphic to
    \[
    \Hom_{G_{n-j}\times G_{j-1}}\left(\beta^{[j]} \boxtimes c-\ind_{U_{j-1}}^{G_{j-1}}(\psi_{j-1}), \ov{r}_{n-j,j-1}(\gamma) \right).
    \]
    We observe that since $j-1<d$, for any irreducible component $\alpha_1\boxtimes\alpha_2$ of $\ov{r}_{n-j,j-1}(\gamma)$, $\rho\in \supp(\alpha_1)$. Since any cuspidal representation in the cuspidal support of any irreducible component of $\beta^{[j]}$ is an unramified twist of a cuspidal representation in $\supp(\beta)$, the Hom-space is zero by comparing the cuspidal support.
\end{proof}

Since $\theta_{n,n}(\pi)\simeq \pi^{\vee}$, we obtain
\[
\Hom_{G_{n-1}}(\theta_{n,n}(\pi),\lambda)=0
\]
by Lemma~\ref{hom-zero lem} and the assumption on the cuspidal representation $\rho$. 

We then prove the Hom-space for big theta is non-zero, namely,
\[
 \Hom_{G_{n-1}}(\Theta_{n,n}(\pi),\lambda)\simeq \Hom_{G_n}(\Theta_{n-1,n}(\lambda)\otimes \Omega_n , \pi)\neq 0.
\]
By Howe duality Theorem~\ref{Howe duality thm}, $\theta_{n-1,n}(\lambda)\simeq \pi^-\times \rho$. Hence, it suffices to show that
\begin{equation}
    \Hom_{G_n}((\pi^-\times \rho)\otimes \Omega_n , \pi)\neq 0.
\end{equation}
since $\theta_{n-1,n}(\lambda)$ is the unique irreducible quotient of $\Theta_{n-1,n}(\lambda)$.
We recall the relevant relation of A-parameters introduced by Gan-Gross-Prasad in~\cite{GGP20} that determines whether this Hom-space is non-zero.
\begin{definition}(relevant relation)\label{relevant def}
    Let $n,m$ be positive integers. Let $\pi_1\in\Irr(G_n)$ and $\pi_2\in \Irr(G_m)$. We call the A-parameter $\phi_1$ of $\pi_1$ is relevant to the A-parameter $\phi_2$ of $\pi_2$ if there exist irreducible representations $\{\tau_i\}_{i=1}^{r+s}$ of $W_F$ with bounded image and positive integers $a_1,\cdots,a_r,b_{r+1},\cdots,b_{r+s},c_{1},\cdots,c_{r+s}$ such that
    \[
    \phi_1= \sum_{i=1}^r \tau_i \boxtimes \Sym^{c_i-1}(\BC^2)\boxtimes \Sym^{a_i-1}(\BC^2)\bigoplus \sum_{i=r+1}^{r+s} \tau_i \boxtimes \Sym^{c_i-1}(\BC^2)\boxtimes \Sym^{b_i-2}(\BC^2) 
    \]
    and 
    \[
    \phi_2=\sum_{i=1}^r \tau_i \boxtimes \Sym^{c_i-1}(\BC^2)\boxtimes \Sym^{a_i-2}(\BC^2)\bigoplus \sum_{i=r+1}^{r+s} \tau_i \boxtimes \Sym^{c_i-1}(\BC^2)\boxtimes \Sym^{b_i-1}(\BC^2).
    \]
\end{definition}

According to this definition, we find that $\pi^-\times \rho$ is relevant to $\pi$. Therefore, the desired non-zero property follows from~\cite[Proposition 5.1, Theorem 1.3]{Ch22}. This completes the proof of Proposition~\ref{reducible prop}.

\subsection{Big theta of repeated exceptional Speh blocks}
This subsection is devoted to the proof of the following Proposition.
\begin{proposition}\label{irreducible prop}
    Let $m_0\geq n_0$ be positive integers. Let $\Pi$ be an Arthur-type representation of $G_{n_0}$ such that
    \[
   \Pi\simeq u(\chi;a,c)^{\times r}
    \]
    where $u(\chi;a,c)$ is exceptional with respect to $(n_0,m_0)$ and $a>1$. Then, $\Theta_{n_0,m_0}(\Pi)$ is irreducible. More precisely, $\Theta_{n_0,m_0}(\Pi)\simeq \mathbf{1}_{m_0-n_0}\times \Pi$.
\end{proposition}
The proof of this proposition follows the strategy of Theorem~\ref{big theta speh thm}, but requires a more careful study of the Hom-space for each graded piece in the rank filtration. Consequently, we focus on computing these Hom-spaces and only sketch the gluing process. 

Before we start the proof, we remark that the exceptional condition is equivalent to $\chi=\mathbf{1}_1$ and the existence of 
$0\leq j<\min\{d,c\}$, such that $z_j =\frac{c-1}{2}-j -\frac{d(c_j-1)}{2}=0$,
  \begin{equation}
    h_j=\frac{c_j-a-(m_0-n_0)+1}{2}\in \BZ \text{ and } 1\leq h_j\leq c_j,
\end{equation}
where $c_j=1+\lfloor\frac{c-1-j}{d}\rfloor$. Let $b=c_j$. Set $n=abr$ and $m=m_0-n_0+abr$. Thus, by Proposition~\ref{separation prop intro}, we can replace $\pi$ by its exceptional part:
\[
\Pi_e=U(\mathbf{1}_1;a,b)^{\times r}
\]
and prove that $\Theta_{n,m}(\Pi_e)$ is irreducible.
  \begin{lemma}\label{graded piece computation lem in irr sec}
         Let $m_0\geq n_0$ be two positive integers, and assume $a>1$. Let $\pi=U(\mathbf{1}_1;a,b)^{\times r} $ be a repeated product of exceptional Speh representations with respect to $(n_0,m_0)$.  Then
         \[
         \Hom_{G_n}( \widetilde{R}_k/\widetilde{R}_{k+1}, \pi)_{\sm}\simeq\begin{cases}
         \mathbf{1}_{m-n}\times \pi   &\text{ when }k=n;\\
        \nu_{m-k}^{\frac{k-n}{2}}\times A\times  U(\mathbf{1}_1;a,b)^{\times r-1}   &\text{ when } k=n-h;\\
         0   &\text{ otherwise } ,
         \end{cases}
         \]
         where $A=Z(\sum_{i=1}^{a-1}[\frac{a-b}{2}-i,\frac{a+b}{2}-1-i] +[\frac{1-m+n}{2},\frac{a+b}{2}-1])$ and
         \[
         h=\frac{b-a-(m-n)+1}{2}.
         \]
    \end{lemma}

Note that $m_0-n_0=m-n$. Thus, equivalently, $\pi$ is  exceptional with respect to $(n,m)$. Before the proof of Lemma~\ref{graded piece computation lem in irr sec}, we recall an irreducibility result of parabolic induction with Speh representations due to Lapid-M\'inguez.
\begin{lemma}[{\cite[Proposition 6.6]{LM16}}]\label{parabolic irr lem}
    Let $\sigma=U(\chi;a,b)=Z(\Sigma_1+\cdots+\Sigma_a)$ be a Speh representation as in~\eqref{Speh Zelevinsky}. And let $\beta=Z(\fkm)\in \Irr$ for some multisegment $\fkm$. Suppose that $\Sigma_1\not\prec \fkm$ and $\fkm\not\prec \Sigma_a$, then $\beta\times \sigma$ is irreducible.
\end{lemma}
Applying this lemma inductively on $r$, we obtain that $A\times U(\mathbf{1}_1;a,b)^{\times r-1}$ is irreducible for any positive integer $r$. 
\begin{proof}[Proof of Lemma~\ref{graded piece computation lem in irr sec}] We focus on $k<n$.
By~\eqref{unitary normalized rank filtration eq}, we have
        \begin{align*}
      &   \Hom_{G_n}( \widetilde{R}_k/\widetilde{R}_{k+1}, \pi)_{\sm}\simeq    \Hom_{G_n}\left(   \ind_{P_k\times \ov{P'_k}}^{G_n\times G_m'}(\rho_k \otimes \widetilde{\eta}_k),\pi\right)_{\sm}\\
      &\simeq \ind_{\ov{P'_k}}^{G_m}\left( \Hom_{L_{k,n-k}}(\rho_k\otimes\widetilde{\eta}_k,\ov{r}_{k,n-k}(\pi))_{\sm}\right).
        \end{align*}
        When $k<n$, then the character $\widetilde{\eta}_k$ on $G_{n-k}$ is $\nu^{\frac{k-m}{2}}$, whose cuspidal support is
        \[
        \left\{ \nu_1^{\frac{2k-m-n+1}{2}}, \nu_1^{\frac{2k-m-n+3}{2}},\cdots, \nu_1^{\frac{n-m-1}{2}} \right\}.
        \]
    In particular, each cuspidal representation occurs once. Thus, by comparing cuspidal supports, if $\alpha\boxtimes \beta$ is a subquotient obtained by applying the geometric lemma to $\ov{r}_{k,n-k}(\pi)$, then
    \[
    \Hom_{G_{n-k}}(\widetilde{\eta}_k,\beta)= 0,
    \]
    unless $\beta$ comes from the Jacquet module of a single Speh representation $U(\mathbf{1}_1;a,b)$. Therefore, the argument in Lemma~\ref{graded piece computation lem} shows that
    \[
   \Hom_{G_{n-k}}(\widetilde{\eta}_k,\ov{r}_{k,n-k}(\pi))\neq 0
    \]
    only if $k=n-h$. Moreover, at this time, $A\times U(\mathbf{1}_1;a,b)^{\times r-1}$ is the unique irreducible representation in its composition series. As representations of $G_k$, we have
    \[
    \Hom_{G_{n-k}}(\widetilde{\eta}_k,\ov{r}_{k,n-k}(\pi))\simeq \Hom_{G_{n-k}}(\widetilde{\eta}_k,r_{n-k,k}(\pi))=\BW_{\widetilde{\eta}_k}(\pi).
    \]
    Then it is socle irreducible by Lemma~\ref{duality commutes induction lem} and Proposition~\ref{big derivative socle irreducible}. In particular, the socle is multiplicity-one in the composition series, which implies that 
    \[
     \Hom_{G_{n-k}}(\widetilde{\eta}_k,\ov{r}_{k,n-k}(\pi))\simeq A\times U(\mathbf{1}_1;a,b)^{\times r-1}
    \]
    since $A\times U(\mathbf{1}_1;a,b)^{\times r-1}$ is irreducible. This completes the proof.
\end{proof}

By a similar gluing argument in Lemma~\ref{glue embedding lem} and the Fourier symmetry, we obtain that
\[
\Hom_{G_n}(\omega_{n,m}, \pi)_{\sm}\hookrightarrow \nu_{m-n+h}^{-\frac{h}{2}}\times A\times U(\mathbf{1}_1;a,b)^{\times r-1}:=\sigma .
\]
and 
\[
\Hom_{G_n}(\omega_{n,m}, \pi)_{\sm}\hookrightarrow  U(\mathbf{1}_1;a,b)^{\times r-1}\times A^{\vee}\times \nu_{m-n+h}^{\frac{h}{2}}.
\]
Then, we analyze the irreducible constituents of $\sigma$. We adopt the notation in Figure~\ref{fig:segments-exceptional-speh} about segments in $A$. By the proof of Theorem~\ref{big theta speh thm} when $a>1$, $\nu_{m-n+h}^{-\frac{h}{2}}\times A$ is of length two with socle $\mathbf{1}_{m-n}\times U(\mathbf{1}_{1};a,b)$ and cosocle
    \[
    Z(\Delta_1+\Delta_2+\sum_{i=1}^{a-1} \Lambda_i).
    \]
   Since the product of two Arthur-type representations is still irreducible,
   \[
  ( \mathbf{1}_{m-n}\times U(\mathbf{1}_{1};a,b))\times U(\mathbf{1}_1;a,b)^{\times r-1} 
   \]
   is irreducible. On the other hand, by Lemma~\ref{parabolic irr lem},
   \[
   Z(\Delta_1+\Delta_2+\sum_{i=1}^{a-1} \Lambda_i)\times U(\mathbf{1}_1;a,b)^{\times r-1}
   \]
   is also irreducible by an inductive argument on $r$. Furthermore,~\cite[Corollary 4.9]{LM16} tells us that $\sigma$ is socle irreducible. Consequently, $\sigma$ is of length two and 
   \[
   \soc(\sigma)= ( \mathbf{1}_{m-n}\times U(\mathbf{1}_{1};a,b))\times U(\mathbf{1}_1;a,b)^{\times r-1} .
   \]
   A similar argument as the proof of Theorem~\ref{big theta speh thm} when $a>1$ shows that 
   \[
   \Hom_{G_n}(\omega_{n,m},\pi)_{\sm}\simeq \soc(\sigma).
   \]
   This completes the proof of Proposition~\ref{irreducible prop}.
  \appendix

\section{A class of parabolic induced representations of length two}
In this appendix, we prove a technical lemma showing that a class of parabolically induced representations has length two. Our tool is basically the Jacquet module and some irreducibility criteria from~\cite{LM16}. The author believes that some more recent progress, such as~\cite{Gu19}, can also be applied to prove it. 

In what follows, we fix a character $\chi=\nu_1^s$ for some $s\in \BC$. The segments involved are always assumed to be based on this cuspidal representation $\chi$.
\begin{proposition}\label{length two prop}
    Let $\Delta_1,\Delta_2$ be two segments such that $\ell(\Delta_1)\geq 1$ and $\Delta_1\prec \Delta_2$. Let $\Lambda_i$ be a family of segments for $1\leq i\leq k$. Assume that
    \begin{enumerate}
        \item $\Lambda_k\prec \Lambda_{k-1}\prec\cdots\prec \Lambda_1\prec \Delta_1\cup \Delta_2$;
        \item $\Delta_1$ is strictly contained in each $\Lambda_i$.
    \end{enumerate}
    Consider $\pi:=Z(\Delta_1)\times Z(\Delta_2+\sum_{i=1}^k \Lambda_i)$. Then,
    \begin{enumerate}
        \item The socle of $\pi$ is isomorphic to 
        \[ \sigma:=Z(\Delta_1\cap \Delta_2)\times Z(\Delta_1\cup \Delta_2+\sum_{i=1}^k\Lambda_i)
        \]
        and the cosocle is isomorphic to 
        \[
       \gamma:= Z(\Delta_1+\Delta_2+\sum_{i=1}^k\Lambda_i ).
        \]
        \item  The length of $\pi$ is two. In other words, there is a short exact sequence
        \[
        0\lra \sigma\lra \pi\lra \gamma\lra 0.
        \]
    \end{enumerate}
\end{proposition}

In order to prove the proposition, we will need a lemma computing the Jacquet module of an irreducible representation whose multisegment has distinct left endpoints. This Lemma generalizes Lemma~\ref{Speh Jacquet lem zelevinsky}.
\begin{lemma}\label{different left endpoint Jac lem}
    Let $\pi=Z(\mathfrak{m})\in\Irr(G_n)$, where $\fkm=\Delta_1+\cdots +\Delta_r$. Suppose that $b(\Delta_i)\neq b(\Delta_j)$ for any $i\neq j$. Let 
    \[
    H(\fkm):=\{\Delta \in \fkm \mid \not\exists \Delta' \in\fkm \text{ such that } b(\Delta)+1=b(\Delta') \text{ and } \Delta\prec \Delta'\}.
    \]
    Then we have
    \begin{equation}\label{Jacquet different left ends}
           r_{1,n-1}(\pi)= \bigoplus_{\Delta\in H(\fkm)} \chi\nu^{b(\Delta)}_1\boxtimes Z(\fkm-\Delta+{}^-\Delta) . 
    \end{equation}
\end{lemma}
\begin{proof}
By Lemma~\ref{reducedness characterization lem}, $\alpha\boxtimes\beta$ is an irreducible component of $r_{1,N-1}(\pi)$ if and only if $\pi$ is not left $\alpha$-reduced. Furthermore, \cite[Theorem~5.11~(4)]{LM16} demonstrates that the irreducible components of $r_{1,N-1}(\pi)$ consist precisely of the terms on the right-hand side of~\eqref{Jacquet different left ends}. Moreover, this decomposition is a direct sum because $\chi\nu^{b(\Delta_i)}\not\simeq \chi\nu^{b(\Delta_j)}$ for any $1\leq i\neq j\leq r$.
\end{proof}
We are ready to prove the Proposition.
\begin{proof}[Proof of Proposition~\ref{length two prop}]
By twisting the character $\chi^{-1}$ to $\pi$, we can assume $s=0$.
   \textbf{Part (1).} We first compute the socle of $\pi$ by induction on the length of $\Delta_1\setminus\Delta_2$. By definition of Zelevinsky segment,
    \[
    Z(\Delta_1)\hookrightarrow Z(\Delta_1\setminus\Delta_2)\times Z(\Delta_1\cap \Delta_2).
    \]
    Therefore,
    \[
     \pi\hookrightarrow Z(\Delta_1\setminus\Delta_2)\times \left(Z(\Delta_1\cap \Delta_2)\times Z(\Delta_2+\sum_{i=1}^k \Lambda_i)\right).
    \]
    Since $\Delta_1\cap \Delta_2$ is contained in $\Delta_2$ and $\Lambda_i$, $\lambda:=Z(\Delta_1\cap \Delta_2)\times Z(\Delta_2+\sum_{i=1}^k \Lambda_i)$ is irreducible by Proposition~\ref{Zele class_properties} part~\ref{part: non-preceding socle}. Thus, it suffices to show that $ Z(\Delta_1\setminus\Delta_2)\times\lambda$ has socle $\sigma$.
    
    Suppose $\ell(\Delta_1\setminus \Delta_2)=1$. Then by direct computation, we observe that it satisfies the assumption in~\cite[Theorem 5.11 (2)]{LM16}. Hence,
    \[
    \soc(Z(\Delta_1\setminus\Delta_2)\times\lambda)= Z(\Delta_1\cap \Delta_2 + \Delta_1\cup \Delta_2+\sum_{i=1}^k \Lambda_i),
    \]
    which is isomorphic to $\sigma$ by Proposition~\ref{Zele class_properties} part~\ref{part: non-preceding socle} again. Suppose that when $\ell(\Delta_1\setminus \Delta_2)<j$, $\soc(Z(\Delta_1\setminus\Delta_2)\times\lambda)=\sigma$. We proceed to assume that $\ell(\Delta_1\setminus \Delta_2)=j$. Then by~\cite[Proposition 5.6]{LM16}, 
    \begin{align*}
        \soc(Z(\Delta_1\setminus\Delta_2)\times\lambda)&\simeq \soc ( \nu_1^{b(\Delta_1) }\times \soc( Z( {}^- (\Delta_1\setminus \Delta_2))\times \lambda))\\
        &\simeq \soc ( \nu_1^{b(\Delta_1)} \times Z(\Delta_1\cap \Delta_2 + {}^- (\Delta_1\cup \Delta_2)+\sum_{i=1}^k \Lambda_i).
    \end{align*}
    Then by~\cite[Proposition 5.6]{LM16} again, we obtain the result.

    We compute the cosocle of $\pi$. Since $\Delta_2\not\prec \Delta_1$ and $\Delta_1$ is contained in $\Lambda_i$ for each $i$, the result follows from Proposition~\ref{Zele class_properties} Part~\ref{part: non-preceding socle}.\\
    \textbf{Part (2).}  Assume that $\pi$ is a representation of $G_N$. We argue by induction on $N$. The minimal possibility is $N=(k+1)(k+2)$, in which case $\ell(\Delta_1)=1$ and $\ell(\Lambda_i)=k+2$ for each $i$. At this time, on the one hand, by geometric lemma and Lemma~\ref{different left endpoint Jac lem}, we have
    \begin{align*}
       & [r_{1,N-1}(\pi)]=[\nu_1^{b(\Delta_1)}\boxtimes Z(\Delta_2+\sum_{i=1}^k \Lambda_i)]+\\
       [\nu_1^{b(\Delta_2)}&\boxtimes Z(\Delta_1)\times Z({}^-\Delta_2+\sum_{i=1}^k \Lambda_i)]+[\nu_1^{b(\Lambda_1)}\boxtimes Z(\Delta_1)\times Z(\Delta_2+{}^-\Lambda_1+\sum_{i=2}^k \Lambda_i)]
    \end{align*}
    in the Grothendieck group. On the other hand, by Lemma~\ref{different left endpoint Jac lem}, we have
    \[
    [r_{1,N-1}(\sigma)]= [\nu_1^{b(\Delta_1)}\boxtimes Z(\Delta_2+\sum_{i=1}^k \Lambda_i)]
    \]
    and
    \[
    [r_{1,N-1}(\gamma)]= [\nu_1^{b(\Delta_2)}\boxtimes Z(\Delta_1)\times Z({}^-\Delta_2+\sum_{i=1}^k \Lambda_i)]+[\nu_1^{b(\Lambda_1)}\boxtimes Z(\Delta_1)\times Z(\Delta_2+{}^-\Lambda_1+\sum_{i=2}^k \Lambda_i)].
    \]
    Since the Jacquet functor is exact, if there is another non-zero irreducible component $\pi'$ besides $\sigma,\gamma$, then $r_{1,N-1}(\pi')=0$. However, it is impossible since the cuspidal support of $\pi'$ consists of characters of $G_1$.

    Now, assume that when $\pi$ is a representation of $G_t$, where $t<N$, the statements holds. We proceed to prove when $\pi$ is a representation of $G_N$. The idea follows from the case when $N=(k+1)(k+2)$. That is, it suffices to show that 
    \[
    \length(r_{1,N-1}(\pi))\leq \length(r_{1,N-1}(\sigma))+\length(r_{1,N-1}(\gamma)).
    \]
    For a finite length representation $V$ of $G_N$, we define
    \[
    \CG(V):=\{x\in \BR\mid \Jac_{\nu_1^x}(V)\neq 0\}.
    \]
    Then by geometric lemma and Lemma~\ref{different left endpoint Jac lem}, we have
    \[
    \CG(\pi)= \{b(\Delta_1),b(\Delta_2),b(\Lambda_1)\}\cup \CY,
    \]
    where $\CY:= \{ b(\Lambda_i)\mid 1<i\leq k \text{ and }
    b(\Lambda_i)+1\neq b(\Lambda_{i-1})\}$. We observe that when $x\in \CG(V)$, $\Jac_{\nu_1^x}(\pi)$ satisfies the assumption of the Proposition except for the following cases:
    \begin{enumerate}
        \item $b(\Delta_1)=b(\Delta_2)-1$, in other words, ${}^-\Delta_1\subset \Delta_2$. In this case, $\Jac_{\nu_1^{b(\Delta_1)}}(\pi)$ is irreducible since ${}^-\Delta_1$ is also contained in each $\Lambda_i$.
        \item $b(\Delta_2)=e(\Delta_1)+1$, equivalently, ${}^-\Delta_2$ is not linked with $\Delta_1$. In this case, $\Jac_{\nu_1^{b(\Delta_2)}}(\pi)$ is irreducible since $\Delta_1$ is also not linked with $\Lambda_i$ for each $i$.
        \item $b(\Lambda_1)=b(\Delta_1)-1$. In this case, $\Jac_{\nu_1^{b(\Lambda_1)}}(\pi)$ is also irreducible by~\cite[Corollary 5.14]{LM16}.
    \end{enumerate}
    Besides these exceptional cases, since $\Jac_{\nu_1^x}(\pi)$ satisfies the assumption of the Proposition, by inductive hypothesis, it is of length two.

    On the other hand, we compute $\CG(\sigma)$ and $\CG(\gamma)$. If we let $\fkm$ be the multisegment $\Delta_1\cap\Delta_2+\Delta_1\cup \Delta_2+\sum_{i=1}^k\Lambda_i$, then by Lemma~\ref{different left endpoint Jac lem}, we have $\CG(\sigma)= \{b(\Delta)\mid \Delta\in H(\fkm)\}$. More explicitly,
    \[
    \CG(\sigma)= \begin{cases}
        \{ b(\Delta_1)\}\cup \CY  & \text{when } b(\Lambda_1)=b(\Delta_1)-1 \text{ and } \Delta_1\not\prec{}^-\Delta_2\\
         \{ b(\Delta_1), b(\Lambda_1)\}\cup \CY &\text{when } b(\Lambda_1)<b(\Delta_1)-1 \text{ and }\Delta_1\not\prec{}^-\Delta_2\\
          \{ b(\Delta_1), b(\Delta_2)\}\cup \CY &\text{when } b(\Lambda_1)=b(\Delta_1)-1 \text{ and }\Delta_1 \prec{}^-\Delta_2\\
         \{ b(\Delta_1),b(\Delta_2), b(\Lambda_1)\}\cup \CY &\text{otherwise } 
    \end{cases}  .
    \]
    In addition, if we let $\fkm'$ be the multisegment $\Delta_1+\Delta_2+\sum_{i=1}^k\Lambda_i$, then by Lemma~\ref{different left endpoint Jac lem}, we have $\CG(\gamma)= \{b(\Delta)\mid \Delta\in H(\fkm')\}$. More explicitly,
    \[
    \CG(\gamma)=\begin{cases}
        \{b(\Lambda_1),b(\Delta_2)\}\cup \CY & \text{when } b(\Delta_1)=b(\Delta_2)-1 \\
         \{b(\Delta_1),b(\Lambda_1),b(\Delta_2)\}\cup \CY & \text{when } b(\Delta_1)<b(\Delta_2)-1 \\
    \end{cases}.
    \]
    In conclusion, the above calculation shows
    \[
    \length(r_{1,N-1}(\pi))= |\CG(\sigma)|+|\CG(\gamma)|
    \]
    which is at most $\length(r_{1,N-1}(\sigma))+\length(r_{1,N-1}(\gamma))$. This completes the proof.
\end{proof}


\begin{thebibliography}{99}
\bibitem[APS17]{APS17} Jeffrey D. Adams, Dipendra Prasad, and Gordan Savin. \emph{Euler-Poincar\'e characteristic for the oscillator representation.} In Representation theory, number theory, and invariant theory, volume 323 of Progr. Math., pages 1–22. Birkha\"user/Springer, Cham, 2017.

\bibitem[Au95]{Au95} Aubert, A.-M.: \emph{Dualit\'e dans le groupe de Grothendieck de la cat\'egorie des repr\'esentations lisses delongueur finie d'un groupe r\'eductif p-adique.} Trans. Am. Math. Soc. 347(6), 2179-2189 (1995) https://doi.org/10.1090/S0002-9947-1995-1285969-0 

\bibitem[Ba07]{Ba07} A. I. Badulescu, Jacquet-Langlands et unitarisabilité. J. Inst. Math. Jussieu 6 (2007), no. 3, 349–379

\bibitem[Ba08]{Ba08} \bysame, Global Jacquet-Langlands correspondence, multiplicity one and classification of automorphic representations. With an appendix by Neven Grbac. Invent. Math. 172 (2008), no. 2, 383–438

\bibitem[Ber92]{Ber92}
J.~Bernstein, \emph{Representations of $p$-adic groups}, notes written by
K.~Rumelhart, Harvard University, 1992.

\bibitem[BBK18]{BBK18} Bernstein, J., Bezrukavnikov, R. and Kazhdan, D. \emph{Deligne-Lusztig duality and wonderful compactification.} Sel. Math. New Ser. 24, 7-20 (2018). https://doi.org/10.1007/s00029-018-0391-5

\bibitem[BR10]{BR10} A. I. Badulescu and D. Renard, \emph{Unitary dual of $\GL_n$ at Archimedean places and global Jacquet-Langlands correspondence}. Compos. Math. 146 (2010), no. 5, 1115–1164

\bibitem[BW00]{BW00} A. Borel, N. Wallach, Continuous Cohomology, Discrete Subgroups, and Representations of Reductive Groups, second ed., Math. Surveys Monogr., vol. 67, Amer. Math. Soc., Providence, RI, 2000.

\bibitem[BZ77]{BZ77} I. N. Bernstein and A. V. Zelevinsky, {\it Induced representations of reductive p-adic groups}, I, Ann. Sci. Ecole Norm. Sup. {\bf 10} (1977), 441-472.

\bibitem[Cai23]{Cai23} Yuanqing Cai, \emph{Quaternionic Speh representations}, Doc. Math. 28 (2023), no. 4, 903–937. MR4705603

\bibitem[CLLTZ25]{CLLTZ25} Rui Chen, Yufeng Li, Xiaohuan Long, Chenhao Tang, Jialiang Zou. \emph{Godement--Jacquet L-function and homological theta lifting.} arXiv preprint arXiv:2507.07531 (2025).


\bibitem[Ch22]{Ch22} Kei Yuen Chan, \emph{Restriction for general linear groups: The local non-tempered Gan-Gross-Prasad conjecture (non-Archimedean case)}. Crelles Journal, vol. 2022, no. 783, 2022, pp. 49-94. 

\bibitem[Ch23]{Cha_qbl} \bysame, \emph{Quotient branching law for p-adic $(\mathrm{GL}_{n+1}, \mathrm{GL}_{n})$ I: generalized GGP relevant pair}. arXiv:2212.05919 (v2, 2023).

\bibitem[Ch25]{Cha_csq} \bysame, \emph{Construction of simple quotients of Bernstein-Zelevinsky derivatives and highest derivative multisegments I: reduction to combinatorics}, Trans. Amer. Math. Soc. Ser. B 12 (2025), 851-909.  


\bibitem[Ch24a]{Cha_csq_ii} \bysame, \emph{Construction of simple quotients of Bernstein-Zelevinsky derivatives and highest derivative multisegments II: minimal sequences}. preprint (2024).



\bibitem[Ch24b]{Cha_csq_iii} \bysame, \emph{Construction of simple quotients of Bernstein-Zelevinsky derivatives and highest derivative multisegments III: properties of minimal sequences}. preprint (2024).

\bibitem[Ch24c]{Cha_duality} \bysame, \emph{Duality for generalized Gan-Gross-Prasad relevant pairs for $p$-adic $\mathrm{GL}_n$}, preprint, arXiv:2210.17249

\bibitem[Ch24]{Ch24} \bysame, \emph{On the product functor on inner forms of general linear group over a non-Archimedean local field}, Transformation Groups (2024).

\bibitem[CP25]{CP25}  K.Y. Chan and Basudev Pattanayak, \emph{Algorithms for parabolic inductions and Jacquet modules in $\mathrm{GL}_n$}, arXiv:2503.00886.

\bibitem[FSX18]{FSX18} Yingjue Fang, Binyong Sun, and Huajian Xue. \emph{Godement-Jacquet L-functions and full theta lifts}. Math. Z., 289(1-2):593–604, 2018.

\bibitem[GGP20]{GGP20} W. T. Gan, B. H. Gross and D. Prasad, \emph{Branching laws for classical groups: the non-tempered case}, Compos. Math. 156 (2020), no. 11, 2298–2367.

\bibitem[GJ72]{GJ72} Roger Godement and Herv\'e Jacquet. \emph{Zeta functions of simple algebras}, volume Vol. 260 of Lecture Notes in Mathematics. Springer-Verlag, Berlin-New York, 1972.

\bibitem[Gu19]{Gu19} Maxim Gurevich. \emph{Decomposition rules for the ring of representations of non-Archimedean $\GL_n$.} International Mathematics Research Notices, 2020(20):6815–6855, 02 2019.

\bibitem[HT01]{HT01} M. Harris and R. Taylor, The geometry and cohomology of some simple Shimura varieties, With an appendix by Vladimir G. Berkovich, Annals of Mathematics Studies, 151. Princeton University Press, Princeton, NJ, 2001.

\bibitem[KL12]{KL12} A. Kret and E. Lapid, \emph{Jacquet modules of ladder representations}, C. R. Math. Acad. Sci. Paris 350 (2012), no. 21-22, 937–940. MR2996769

\bibitem[LM14]{LM14}
E. Lapid and A. M\'inguez, \emph{On a determinantal formula of Tad\'ic}, Amer. J. Math. 136 (2014), no. 1, 111–142. MR3163355.

\bibitem[LM16]{LM16} \bysame, \emph{On parabolic induction on inner forms of the general linear group over a non-Archimedean local field}, Sel. Math. New Ser. (2016) 22, 2347-2400.

\bibitem[Ming08]{Ming08}
A.~Mínguez, \emph{Correspondance de Howe explicite: paires duales de type
II}, Ann. Sci. Éc. Norm. Supér. (4) \textbf{41} (2008), no.~5, 717--741.
DOI: 10.24033/asens.2080.

\bibitem[Ming09]{Ming09}
\bysame, \emph{Sur l'irréductibilité d'une induite parabolique},
J. Reine Angew. Math. \textbf{629} (2009), 107--131.
DOI: 10.1515/CRELLE.2009.028.

\bibitem[MS13]{MS13} M\'inguez, A., S\'echerre, V. (2013). Repr\'esentations banales de ${\mathrm GL}_{m}({\mathrm D})$. Compositio Mathematica, 149(4), 679-704. doi:10.1112/S0010437X12000590

\bibitem[Pr23]{Pr23} Dipendra Prasad. \emph{Homological aspects of branching laws}. arXiv preprint arXiv:2302.03492, 2023.

\bibitem[SS97]{SS97} P. Schneider, U. Stuhler, \emph{Representation theory and sheaves on the Bruhat-Tits building}, Publ. Math. Inst. Hautes \'Etudes Sci. 85 (1997) 97-191.

\bibitem[GW26]{GW26} Z. Geng and K. Wu, \emph{Multiplicity one property for Archimedean theta correspondence}, arXiv preprint, 	arXiv:2609.26630 (2026).

\bibitem[Ze80]{Ze80} A. Zelevinsky, \emph{Induced representations of reductive p-adic groups II}, Ann. Sci. Ecole Norm. Sup. {\bf 13} (1980), 154-210.
\end{thebibliography}
\end{document}